\documentclass[11pt,twoside]{article}
\usepackage{amsmath,amsfonts,amssymb,a4,enumitem,epsfig,array,psfrag}
\usepackage{url}
\usepackage{amsthm}
\usepackage{color}
\usepackage{calc}
\usepackage{etoolbox} % For qed at end of example
\newtheorem{prop}{Proposition}[section]

\newtheorem{lemma}[prop]{Lemma}

\newtheorem{theo}[prop]{Theorem}
\newtheorem{theo-def}[prop]{Theorem-Definition}

\theoremstyle{definition}
\newtheorem{remark}[prop]{Remark}
\newtheorem{definition}[prop]{Definition}
\newtheorem{example}[prop]{Example}

\AtEndEnvironment{example}{\null\hfill$\diamond$}%
\AtEndEnvironment{remark}{\null\hfill$\diamond$}%
\AtEndEnvironment{definition}{\null\hfill$\diamond$}%

\newcommand{\bt}{{\mathbf t}}
\newcommand{\bq}{{\mathbf q}}

\newcommand{\sat}{\operatorname{sat}}

\newcommand{\Flux}{\operatorname{Flux}}
\newcommand{\RFlux}{\operatorname{RFlux}}
\newcommand{\Hel}{\operatorname{Hel}}
\newcommand{\Tw}{\operatorname{Tw}}
\newcommand{\lk}{\it{lk}}
\newcommand{\slk}{\it{slk}}
\newcommand{\supp}{\operatorname{supp}}

\newcommand{\new}{\operatorname{new}}

\newcommand{\NN}{{\mathbb{N}}}

\newcommand{\RR}{{\mathbb{R}}}
\newcommand{\TT}{{\mathbb{T}}}
\newcommand{\Ss}{{\mathbb{S}}}
\newcommand{\ZZ}{{\mathbb{Z}}}

\newcommand{\cR}{{\cal R}}

\newcommand{\cT}{{\cal T}}

\usepackage[bottom=2cm, right=2.6cm, left=2.6cm, top=1.9cm]{geometry}

\begin{document}
\title{Relative Helicity and Tiling Twist}
\author{Boris Khesin \and Nicolau C. Saldanha}
\date{}
\maketitle

\begin{abstract}
We consider domino tilings of 3D cubiculated regions.
The tilings have two invariants, flux and twist,
often integer-valued, which are given in purely combinatorial terms.
These invariants allow one to classify the tilings
with respect to certain elementary moves, flips and trits.
In this paper we present a construction associating
a divergence-free vector field $\xi_\bt$ to any domino tiling $\bt$,
such that the flux of the tiling  
$\bt$ can be interpreted
as the (relative) rotation class of the field  $\xi_\bt$,
while  the twist of  $\bt$ is proved to be
the relative helicity of the field $\xi_\bt$. 

%A \textit{flip} is a local move: remove two neighboring parallel dominoes and place them back after a rotation.
%The \textit{twist} of a domino tiling is an integer, invariant under flips.
% Several definitions of the twist are known, none of them trivial. In this papers we construct, for each tiling $\bt$,
%a divergence-free vector field $\xi_\bt$. The twist of $\bt$ is then given, up to a multiplicative factor, by the relative helicity if $\xi_\bt$.
\end{abstract}

%%%%%%%%%%%%%%%%%%%%%%%%%%%%%%%%%%%%%%%%%%%%%%%%%%%%%%%%%%%%%%%%%%%%%%%

\footnotetext{2010 {\em Mathematics Subject Classification}.
Primary 05B45; Secondary 57K12, 52C22, 05C70.
{\em Keywords and phrases:} Three-dimensional tilings,
dominoes, dimers, flux, helicity, rotation class, relative helicity.}

\bigbreak

%\tableofcontents

\section{Introduction}

Domino tilings of 3D  regions traditionally have
two invariants associated to them, flux and twist.
Flux is understood as the homology class % of the region 
associated with a certain cycle constructed for the  tiling.
For a  cubiculated region   (a topological $3$-manifold with boundary)
 the flux $\Flux(\bt)$ of a tiling $\bt$
is an element  of the first homology group of the region,
defined up to an additive constant.
The ambiguity can be removed by considering
the \textit{relative flux} $\RFlux(\bt)$ % \in H_1(\cR,\partial\cR;\RR)$
as discussed below.
For tilings with zero relative flux
% zero flux (and in certain other cases)
there is an integer invariant, the \textit{twist}, 
associated with the tiling
and measuring the `mutual linking of tiles' around each other.
It is invariant with respect to \textit{flips}, 
local moves which consist of removing two neighboring parallel dominoes
and placing them back after a rotation.
The twist changes under another move, a \textit{trit},
which replaces a frame-like triple of tiles
to the one pointing in the opposite way
(see Section~\ref{section:domino} and   a detailed discussion in \cite{FKMS}).

While there have been pointed out similarities of the twist invariant
with the Hopf invariant,
the correspondence remained at either an intuitive level or
in the  continuous  limit for turning tiles into vector fields
via a broadly understood tiling's refinement.

In this paper we present a construction of
a smooth divergence-free vector field
associated to an arbitrary tiling
(`5-pipe construction', see Section~\ref{section:fivepipes}) so that
\textit{the twist invariant becomes, up to a factor, the relative helicity}
of that vector field.
Furthermore, we extend the notion of relative helicity \cite{BF}
to divergence-free vector fields on arbitrary three-manifolds and not necessarily tangent to their boundaries.
This allows one to compare relative helicity with twists of tilings
in non-simply-connected regions (see Section~\ref{section:relhel}).
The toolbox includes an introduction of `an isolating shell'
for a cubiculated region,
the use of refinements and appropriate connectivity of the spaces of tilings. 
Finally, we relate \textit{the flux invariant of a tiling
to the rotation class} of the associated vector field. 

\medskip

In a nutshell, the results of the paper are as follows.
Recall that for a divergence-free vector field $\xi$
in a simply-connected domain $M\subset \RR^3$
and tangent to its boundary $\partial M$,
its helicity is the quantity 
$\Hel(\xi)= \int_{M} (\xi , {\rm{curl}}^{-1}\xi )\; d^3x$.
For a field whose support consists of several linked tubes
the helicity reduces to the mutual linking of those tubes,
the self-linking and fluxes in the tubes
(see \cite{Moffatt69} and Section~\ref{section:helipipe} and Appendix).

The helicity notion can be extended to null-homologous vector fields
tangent to the boundary in arbitrary three-manifolds $M$
equipped with a volume form.
If the field $\xi$ is not tangent to the boundary $\partial M$,
one can define only its relative helicity,
i.e., the difference of helicities of two vector fields 
$\xi$ and $\eta$ with the same behaviour at the boundary $\partial M$.
Namely, one can extend $\xi$ and $\eta$ by the same field
to vector fields in a bigger manifold $\widetilde M\supseteq M$,
where their helicities are well-defined, but depend on the extension.
However, the difference of helicities will not depend on the extension,
hence the name \textit{relative helicity}, see Section \ref{section:relhel}.

On the other hand, given a 3D domino tile, consisting of two cubes
(say, white and black),
consider five smooth curves joining the  symmetric faces at their centers
and approaching those faces orthogonally, see Figure~\ref{fig:5pipe}.
Tubular neighborhoods of those curves will become supports
of the five smooth divergence-free vector fields
directed from the black to the white unit cube and each having flux $\varphi$.
By performing this construction in each domino,
to any tiling $\bt$ of a cubiculated region $\cR$ one associates a smooth
divergence-free vector field $\xi_\bt$ in the whole region.
% for simplicity that the tiling $\bt$ has zero flux
% (this condition will be weakened later, see  Chapter ??). 
By using an `isolating shell'  for the region,
one can define the relative helicity of the vector field $\xi_\bt$
associated with the tiling.
The construction of such an isolating shell turns out to be possible
whenever the relative flux vanishes,
$\RFlux(\bt) = 0\in H_1(\cR, \partial \cR)$.
Our main result is the following:

\begin{theo}
\label{theo:main}
% \begin{theo}[{\bf(=Theorem \ref{}$'$)}]
For  two tilings $\bt_0$ and $\bt_1$ of the same flux and of zero relative flux in a 3D cubiculated region $\cR$,
the difference of their twists is proportional to the relative helicity of the associated vector fields $\xi_{\bt_0}$ and
$\xi_{\bt_1}$:
\[
\Hel(\xi_{\bt_1})- \Hel(\xi_{\bt_0})=36 \varphi^2 \left(\Tw(\bt_1)  - \Tw(\bt_0)\right)\,,
\]
where $\varphi \in \RR$ is
the flux in a single pipe of the vector fields $\xi_{\bt_i}$.
%The constant $C \in \RR$ depends only on $\bt_0$ and on the choice of the isolating shell.
\end{theo}

%As we discussed above,  the relative helicity is actually well-defined as the difference of helicities of two vector fields,
% hence the additive constant $C$.
%  in the above statement  $\Hel(\xi_\bt)$ for a given tiling $\bt$
% is defined up to an additive constant
% depending only on the isolating shell for $\cR$.

% Below we also extend this theorem to include more general,
% i.e. nonsimply-connected,
% regions $\cR$ with some types tilings with nonzero flux. 

In order to clarify the condition of vanishing relative flux,
recall that given a  vector field $\xi$ in a manifold $M$
with a volume form $\mu$ and tangent to the boundary $\partial M$,
the 2-form $\omega_\xi:=i_\xi\mu$ is closed iff
$\xi$ is divergence-free with respect to $\mu$. 
If $\xi$ is tangent to $\partial M$,
the restriction of the 2-form $\omega_\xi$
vanishes on the boundary $\partial M$.
The cohomology class $[\omega_\xi]\in H^2(M, \partial M)$
can be identified with a class in $H_1(M)$ (via the Poincar{\'e} isomorphism)
and is called the {\it rotation class} 
of the dynamical system $\xi$, see \cite{Arnold69}.
Similarly, if a divergence-free vector field $\xi$
is not tangent to $\partial M$, then $[\omega_\xi]\in H^2(M)
\simeq H_1(M, \partial M)$.
Thus such a vector field $\xi$ defines a {\it relative rotation class}
as an element of $H_1(M, \partial M)$.

In Section \ref{sect:rotation} we prove the relation of flux
$\Flux(\bt)$ and relative flux $\RFlux(\bt)$ for a tiling $\bt$
to the (relative) rotation class of an appropriate vector field $\xi_\bt$:

\begin{theo}
\label{theo:main2}
% \begin{theo}[{\bf(=Theorem \ref{}$'$)}]
Given a tiling $\bt$ of a 3D cubiculated region $\cR$ 
its relative flux $\RFlux(\bt)\in H_1(\cR, \partial \cR)$
coincides modulo a factor with 
the relative rotation class $[\xi_\bt]$
of the vector field $\xi_\bt$  obtained via the 5-pipe construction:
$$
[\xi_\bt]=6\varphi \RFlux(\bt)\,.
$$
\end{theo}

%Delete?:  For a tiling $\bt$ in $\cR$ supplemented with an isolating shell,
%one can consider the rotation class of the field $\xi_\bt$ obtained via the 5-pipe construction.
In particular, the relative rotation class of the  field $\xi_\bt$  vanishes
iff the relative flux $\RFlux(\bt)$ vanishes.
By normalizing the flux as $\varphi=1/6$,
consistent with the 6-pipe setting of Section \ref{sect:rotation},
the two main theorems provide the exact matching:
for any tiling its relative flux coincides with
the relative rotation class of the associated vector field,
while the relative helicity of the latter coincides with the tiling's  twist.

\medskip

\begin{remark}
There are numerous papers studying tilings
in the context of  quantum dimer models.
In the papers \cite{FHNQ, Bednik}
the authors emphasized the existence of
topologically non-trivial configurations called `Hopfions' in those  systems
with the understanding that in a suitable continuous limit of the dimer model,
those excitations 
would correspond to the Hopf map $\Ss^3\to \Ss^2$.
In \cite{Arnold73} (see also \cite{AK}),
Arnold proved that an asymptotic version of
the Hopf invariant for a vector field is equal to the field's helicity.
Theorem \ref{theo:main} shows that even without taking the continuous limit,
a tiling's twist itself is already equal (modulo the flux factor)
to the corresponding relative helicity.

In a sense, the present paper shows that
one can regard tilings as occupying an intermediate place between 
divergence-free vector fields on the one hand
and links with framings on the other:
by associating special divergence-free vector fields to tilings,
one can compute their helicities
via consideration of linking and self-linking of framed cycles. 
\end{remark}

The paper is organized as follows.
In Sections~\ref{section:domino} and \ref{section:twist}
we recall the main moves and invariants for domino tilings
in purely combinatorial terms.
In Section~\ref{section:rotation} we introduce 
relative rotation class and relative helicity
of vector fields not tangent to  boundary.
In Section~\ref{sect:all-pipes} we present the 5-pipe construction
associating a divergence-free vector field to a domino tiling,
introduce an isolating shell, and prove Theorem \ref{theo:main2}
on the relative rotation class. 
Section~\ref{section:firstexamples} presents
two key examples which are used in the proof of the main
Theorem~\ref{theo:main} in Section \ref{section:proofmain}.
One more example manifesting the main theorem
is described in Section \ref{section:extra}.
Appendix~\ref{appendix} derives
the formula for helicity via linking and self-linking numbers of pipes.

\medskip

{\bf Acknowledgements.} We are grateful to the anonymous referee for useful remarks.
B.K.  is indebted to  PUC-Rio and IMPA in Rio de Janeiro and IHES in Bures-sur-Yvette for their kind hospitality. 
He was  partially supported by an NSERC Discovery Grant. 
N.S. is thankful for the generous support of CNPq, CAPES and FAPERJ (Brazil).

%%%%%%%%%%%%%%%%%%%%%%%%%%%%%%%%%%%%%%%%%%%%%%%%%%%%%%%%%%%%%%%%%%%%%%%

\section{Domino tilings, combinatorial flux and local moves}
\label{section:domino}

A cubiculated region is a connected and oriented manifold $\cR$
(usually with boundary) decomposed into finitely many unit cubes.
A simple example of cubiculated region is a \textit{box}:
$\cR = [0,L] \times [0,M] \times [0,N] \subset \RR^3$,  $LMN$ even;
the unit cubes are
$[a,a+1] \times [b,b+1] \times [c,c+1]$, $(a,b,c) \in \ZZ^3$.
We assume that the interior of a cubiculated region
is as in this example, so that interior edges (resp. vertices)
are surrounded by four (resp. eight) unit cubes.
We also assume that unit cubes are painted black and white,
with adjacent cubes of opposite colors
and the same number of cubes of each color.
A (3D) domino is the union of two unit cubes with a common face,
thus a $2\times 1\times 1$ rectangular cuboid.
A (domino) tiling of $\cR$ is a set of dominoes with disjoint interiors
whose union is equal to $\cR$.

\begin{figure}[ht]
\centering
\def\svgwidth{75mm}
\includegraphics[scale=1.0]{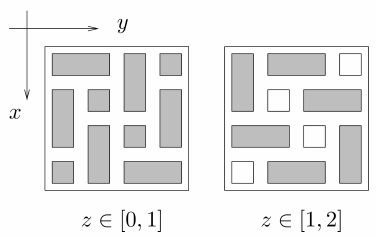}
\caption{A tiling of the box $[0,4]\times [0,4]\times [0,2]$.
% Notice that the position of the axes is
% different from \cite{primeiroartigo}.
The orientation of $\RR^3$ is  important:
the $z$ axis points upward away from the paper.
Examples of dominoes in this tiling are
$[0,1]\times[0,2]\times[0,1]$, $[0,2]\times[0,1]\times[1,2]$
and $[1,2]\times[1,2]\times[0,2]$.
This tiling admits no flips.}
\label{fig:twist2}
\end{figure}

We follow \cite{primeiroartigo}, \cite{saldanha2022} and \cite{FKMS}
in drawing tilings of cubiculated regions by floors,
as in Figure \ref{fig:twist2}.
Vertical dominoes
(i.e., dominoes in the $z$ direction)
appear as two squares, one in each of two adjacent floors;
the top square of a vertical domino
(which, in the figure, appears at the right)
is left unfilled for visual facility.

\begin{definition}
Given a region $\cR$, let $\cT(\cR)$
denote the set of domino tilings of $\cR$.
A \textit{flip} is a local move in $\cT(\cR)$:
two parallel and adjacent (3D) dominoes are removed
and placed back in a different position.
A \textit{trit} is the only local move in $\cT(\cR)$
involving three dominoes which does not reduce to flips.
% (see \cite{primeiroartigo}, \cite{FKMS}).
The three dominoes involved are in three different directions
and fill a $2\times 2\times 2$ box minus two opposite unit cubes.
\end{definition}

Figure \ref{fig:trit} shows in the $3\times 3\times 2$ box
a trit followed by a sequence of flips (the first tiling admits no flips).
We shall come back to this example in Section~\ref{section:extra}.

\begin{figure}[ht]
\begin{center}
\includegraphics[scale=0.25]{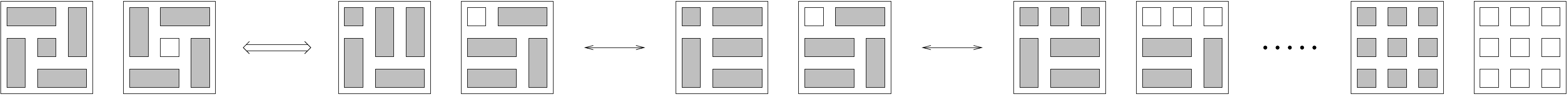}
\end{center}
\caption{Five tilings of the box $[0,3]\times[0,3]\times[0,2]$.
The first move is a trit, the two other moves are flips.
Five more flips take us from the fourth to the fifth tiling.
The first tiling has twist $-1$, the others have twist $0$.}
\label{fig:trit}
\end{figure}

%\footnote{
%Flip and trit are special cases of a more general move, a \textit{cycle} (sometimes called a \textit{tiling cycle} to avoid confusion with other related uses of the word in homology theory, graph theory or group theory).
%
%Two distinct domino tilings $\bt_0, \bt_1 \in \cT(\cR)$ differ by a {\it cycle of length} $2\ell \ge 4$ if there exist
%$(2\ell)$-cyclically numbered distinct unit cubes $c_0, c_1, \ldots, c_{2\ell - 1} \subset \cR$ with the following properties.
%For any $k$, the cubes $c_k$ and $c_{k+1}$ are adjacent. The dominoes $c_{2k}c_{2k+1}$ are in $\bt_1$;
% the dominoes $c_{2k-1}c_{2k}$ are in $\bt_0$. Elsewhere, the tilings $\bt_0$ and $\bt_1$ coincide.
%Notice that colors of cubes alternate; we assume that $c_0$ is black.
%
%A flip is a (tiling) cycle of length $4$ and a trit is a cycle of length $6$.
%The first and last tilings in Figure~\ref{fig:trit} differ by a cycle of length $16$.
% In general, any two tilings $\bt_0, \bt_1 \in \cT(\cR)$  differ by finitely many disjoint cycles,
% as can be seen by superimposing $\bt_0$ and $\bt_1$. }

To apply homology theory we will  use a finer complex structure $\cR^\sharp$.
The vertices of $\cR^\sharp$ are the original vertices of $\cR$,
together with the centers of edges, faces and unit cubes of $\cR$.
Each edge of $\cR$ is decomposed into two edges of $\cR^\sharp$.
Similarly, each square (resp. cube) of $\cR$
is decomposed into four squares (resp. eight cubes) of $\cR^\sharp$.
Notice that the boundary $\partial\cR$ also acquires
the structure of a polyhedral complex  $\partial\cR^\sharp$.

A domino is identified with the sum of two oriented edges in $\cR^\sharp$,
from the center of the black cube to the center of the common face
and from there to the center of the white cube.
A domino tiling $\bt$ of $\cR$ is therefore identified
as $\bt \in C_1(\cR^\sharp)$.
The boundary $\partial\bt \in C_0(\cR^\sharp)$ is
the sum of the centers of the white cubes (of $\cR$)
minus
the sum of the centers of the black cubes.
In particular, for any two tilings $\bt_0, \bt_1$ we have
$\partial\bt_0 = \partial\bt_1$ or, equivalently,
$\bt_1 - \bt_0 \in Z_1(\cR^\sharp)$.
% A cycle (in the above sense) defines a cycle in the homological sense:
% an oriented simple closed curve $\gamma$ going from the center of $c_0$
% to the center of $c_1$ and so on until coming back to the center of $c_0$.
We then write  $[\bt_1 - \bt_0]  \in H_1(\cR) = H_1(\cR;\RR)$;
in the present paper we almost always use coefficients in $\RR$.

% notice that in the present paper we always use coefficients in $\RR$.

%Given a base tiling $\bt_{\oplus} \in \cT(\cR)$,
%one defines the \textit{flux} of a tiling $\bt \in \cT(\cR)$ by
%\begin{equation}
%\label{equation:Flux}
%\Flux(\bt) = [\bt - \bt_{\oplus}] \in H_1(\cR) = H_1(\cR;\RR). 
%\end{equation}
%In many examples, the choice of the base tiling $\bt_{\oplus}$,
%or at least its equivalence class under homology,
%can be considered natural.

Given an initial tiling $\bt_0$,
%and given $\bff \in  H_1(\cR)$,
let $\cT_0(\cR) \subseteq \cT(\cR)$
be the set of tilings $\bt$ belonging to the same `homology class',
i.e. such that 
$[\bt - \bt_0] = 0 \in H_1(\cR)$.
We will say that tilings $\bt$ have the same \textit{flux} 
as $\bt_0$, $\Flux(\bt) = \Flux(\bt_0) $.

In the present paper we define a somewhat different object,
the \textit{relative flux} $\RFlux(\bt) \in H_1(\cR,\partial\cR)$.
It takes values in relative homology
and does not require the choice of a base tiling.
We work in the complex $\cR^\sharp$ and,
given a tiling $\bt$, define $\bt \in C_1(\cR^\sharp)$ as above.
% Recall that $\partial\bt \in C_0(\cR^\sharp)$ does not depend 
% on the choice of $\bt$.
We construct $\bq_1 \in C_1(\cR^\sharp)$ as follows:
 it is the sum over all  black unit cubes of all edges from the cube centers 
to the centers of their faces minus the same sum (of all edges from cube centers to their faces) over all white cubes.
%\footnote{BK: previous version: "for every edge $\varepsilon$ from the center of a black (resp. white) unit cube 
%to the center of a face of the cube, we add (resp. subtract) $\varepsilon$."}}
Thus, for every tiling $\bt$ we have that
$\bq_0 = \partial(6\bt - \bq_1) \in C_0(\partial\cR^\sharp)$
is the sum of $+1$ (resp. $-1$) times
the centers of the white (resp. black) squares in $\partial\cR$.
Thus, $6\bt - \bq_1 \in Z_1(\cR^\sharp,\partial\cR^\sharp)$;
% Since in the present paper we are more interested in real coefficients,
we define
\begin{equation}
\label{equation:RFlux}
\RFlux(\bt) = \left[\bt - \frac{\bq_1}{6} \right] \in H_1(\cR,\partial\cR). 
\end{equation}
Inclusion induces a homomorphism $i: H_1(\cR) \to H_1(\cR,\partial\cR)$, while
for any tiling $\bt$ one has 
$i(\Flux(\bt)) = \RFlux(\bt) - \RFlux(\bt_0)$.

\bigbreak

\begin{example}
\label{example:cubiculatedregions}
Already for cubiculated regions $\cR \subset \RR^3$,
there are examples with nontrivial homology or homotopy.
The solid torus
\[ \cR_0 = ([0,12] \times [0,12] \times [0,4]) \smallsetminus
((4,8) \times (4,8) \times [0,4]) \]
is not simply connected, with $\pi_1(\cR_0) = \ZZ$,
$H_1(\cR_0;\RR) = \RR$ and $H_1(\cR_0,\partial\cR_0;\RR) = 0$.
There exist tilings of $\cR_0$ with different values of flux,
but the relative flux is always $0$.

The cube with a hole
\[ \cR_{1,0} = ([0,12] \times [0,12] \times [0,12]) \smallsetminus
((4,8) \times (4,8) \times (4,8)) \]
is simply connected but has nontrivial relative homology:
$H_1(\cR_{1,0},\partial\cR_{1,0}) = \RR$.
For any tiling $\bt \in \cT(\cR_{1,0})$ we have $\RFlux(\bt) = 0$.
For a different cube with hole
\[ \cR_{1,1} = ([0,13] \times [0,13] \times [0,13]) \smallsetminus
((4,9) \times (4,9) \times (4,9)) \]
we have $\RFlux(\bt) \ne 0$ for all $\bt \in \cT(\cR_{1,1})$,
since the numbers of white and black cubes
in the hole surrounded by $\cR_{1,1}$ are different, see \cite{FKMS}.

Quotients and similar identifications
give us other interesting examples of regions.
The quotient
\[ \cR_2 = \RR^3 / (6\ZZ)^3 \]
is a cubiculated region homeomorphic to a $3$-torus $\TT^3 = (\Ss^1)^3$
(see Figures~2, 7 and 8 in \cite{FKMS} for examples of tilings of $\cR_2$);
since $\partial\cR_2 = \emptyset$,
we have $H_1(\cR_2) = H_1(\cR_2,\partial\cR_2)$.

The quotient
\[ \cR_3 = (\RR^2/ (8\ZZ)^2)\times [0,4]\]
is a cubiculated region homeomorphic to $\TT^2 \times [0,1]$,
so that $H_1(\cR_3) = \RR^2$. It admits tilings with  different values of the flux.
We have $H_1(\cR_3,\partial\cR_3) = \RR$, while the map
$i: H_1(\cR_3) \to H_1(\cR_3,\partial\cR_3)$ is the zero homomorphism, $i = 0$
and $\RFlux(\bt) = 0$ for all $\bt \in \cT(\cR_3)$.
% The two tilings of $\cR_3$ shown in Figure~\ref{fig:tori}
% ($\cR_3$ is as in Example~\ref{example:cubiculatedregions})
% have different flux but the same relative flux. 
% see Figure~\ref{fig:tori} for examples of tilings of $\cR_3$.

We may also remove a few cubes from a larger region;
for instance,
\( \cR_4 = \cR_2 \smallsetminus (2,4)^3 \)
is a $3$-torus with a hole.
\end{example}

\begin{remark}
The examples above demonstrate independence of the assumptions of zero relative flux and two tilings having the same flux.
Indeed,   for $\cR_2 $  the relative and absolute homology groups coincide and one can have two tilings with the same non-zero flux,  $[\bt_1]=[ \bt_0] \not=0 \in H_1(\cR_2) = H_1(\cR_2;\partial\cR_2)$.
On the other hand, all  tilings of $\cR_3 $ have zero relative flux, $\RFlux(\bt) = 0$, while realizing various fluxes in $H_1(\cR_3)= \RR^2$.
\end{remark}

%%%%%%%%%%%%%%%%%%%%%%%%%%%%%%%%%%%%%%%%%%%%%%%%%%%%%%%%%%%%%%%%%%%%%%%

\section{Twist of a tiling}
\label{section:twist}

Recall that given an initial tiling  $\bt_0$ we denote by $\cT_0(\cR)$
 the set of tilings $\bt$ with 
$\Flux(\bt) = [\bt - \bt_0] =0\in H_1(\cR)$.
In general, the {\em twist} of a tiling
% (see \cite{primeiroartigo}, \cite{FKMS}, \cite{saldanha2022})
is a map $\Tw: \cT_0(\cR) \to \ZZ/(m\ZZ)$
where $m \in \ZZ$, $m \ge 0$.
In the present paper we are  concerned with the case  $m = 0$, 
where the twist map is a function $\Tw: \cT_0(\cR) \to \ZZ$ 
well-defined up to an additive constant.
In this case, one can define the twist of a tiling recursively and,
if needed, using refinements as follows
(we refer to \cite{FKMS} for a general combinatorial definition
of the twist and more details).

%We present properties of the twist sufficient to proceed with our constructions and results.
% If two tilings $\bt_0, \bt_1 \in \cT_{0}(\cR)$ are joined by a flip
% then $\Tw(\bt_1) = \Tw(\bt_0)$.
% Similarly, if $\bt_0$ and $\bt_1$ are joined by a trit
% then $\Tw_{\bff}(\bt_1) = \Tw_{\bff}(\bt_0) \pm 1$;
% sign is given by the orientation of the trit.
Given a cubiculated region $\cR$,
the \textit{refinement} $\cR'$ of $\cR$
is obtained by decomposing each unit cube in $\cR$
into $5\times 5\times 5$ smaller cubes.
Given a tiling $\bt$ of $\cR$,
its refinement $\bt'$ is a tiling of $\cR'$:
each domino $d$ in $\bt$ is decomposed 
into $5\times 5\times 5$ smaller dominoes, all parallel to $d$.
We write $\cR^{(0)} = \cR$ and $\cR^{(k+1)} = (\cR^{(k)})'$
and define $\bt^{(k)} \in \cT(\cR^{(k)})$ in a similar manner.
%The following is one of the main results in \cite{FKMS}.

\begin{theo} {\rm (\cite{FKMS})}
\label{theo:flipsandtrits}
Consider a cubiculated region $\cR$
and two tilings $\bt_0, \bt_1$ of $\cR$.
If $\Flux(\bt_1) = \Flux(\bt_0)$ 
then there exist $k$ such that
the tilings $\bt_0^{(k)}$ and $\bt_1^{(k)}$
can be joined by a finite sequence of flips and trits.
\end{theo}

Define the set of all refinements of $\cT_0(\cR)$ by 
\[ \cT_0(\cR^{(\ast)}) =
\bigsqcup_{k \in \NN} \cT_0(\cR^{(k)}) \]
where $\cT_0(\cR^{(k)}) \subseteq \cT(\cR^{(k)})$
is the subset of tilings $\bt$ with $[\bt - \bt_0^{(k)}] = 0 \in H_1(\cR)$.
% which is nonempty for all $\bt_0$.
One can interpret  $\cT_0(\cR^{(\ast)})$ as the set of vertices
of an infinite graph:
two vertices $\bt_0$ and $\bt_1$ are connected by an edge if the tilings 
$\bt_0$ and $\bt_1$ differ by a flip, a trit or by one step of a refinement. The graph $\cT_0(\cR^{(\ast)})$ is connected, as 
follows directly from Theorem~\ref{theo:flipsandtrits}.
% The following result %, also from \cite{FKMS},
% gives us a recursive definition of twist in $\cT_0(\cR)$.

\smallskip

The following result %, also from \cite{FKMS},
gives us a recursive definition of twist in $\cT_0(\cR)$. It is proved in \cite{FKMS} in different terms
and (in certain cases before that) in \cite{segundoartigo}. Below we provide a different proof.

\begin{theo-def} {\rm (cf. \cite{FKMS})}
\label{theo:Fluxtwist}
 Consider a cubiculated region $\cR$, with an initial tiling 
$\bt_0$. If the relative flux of $\bt_0$ is zero then
there exists a function $\Tw: \cT_0(\cR^{(\ast)}) \to \ZZ$
(unique up to an additive constant) with the following properties:
\begin{enumerate}
\item{If two tilings $\bt_a, \bt_b \in \cT_0(\cR^{(\ast)})$
are joined by a flip then $\Tw(\bt_b) = \Tw(\bt_a)$.}
\item{If  $\bt_a$ and $\bt_b$ are joined by a trit
then $\Tw(\bt_b) = \Tw(\bt_a) \pm 1$;
sign is given by the orientation of the trit.}
\item{If $\bt'$ is the refinement of $\bt \in  \cT_{0}(\cR^{(\ast)})$
then $\Tw(\bt') = \Tw(\bt)$.}
\end{enumerate}
The function $\Tw$ is called the {\rm twist function}.
\end{theo-def}

% \item{If $m > 0$ then there exist a finite sequence of
% tilings $\bt_0, \bt_1, \ldots, \bt_\ell = \bt_0$ in $\cT_{0}(\cR^{(k)})$
% (for some $k \ge 0$)
% and a sequence of integers $t_0, t_1, \ldots, t_{\ell}$
% with the following properties.
% For every $i$, $\Tw(\bt_i) = t_i \bmod m$.
% For every $i$, $\bt_i$ and $\bt_{i+1}$ are joined by a flip or trit.
% If $\bt_i$ and $\bt_{i+1}$ are joined by a flip then $t_{i+1} = t_{i}$.
% If $\bt_i$ and $\bt_{i+1}$ are joined by a trit
% then $t_{i+1} = t_i \pm 1$; sign is given by the orientation of the trit.
% We have $t_\ell = t_0 + m$.}

We say that a trit from $\bt_a$ to $\bt_b$ is
\textit{positive} (respectively, \textit{negative})
if $\Tw(\bt_b) - \Tw(\bt_a) = +1$ (respectively, $-1$).

\begin{remark}
If $\RFlux(\bt_0)\not=0$, there is $m>0$ such that the twist function  becomes 
a map $\Tw: \cT_0(\cR) \to \ZZ/(m\ZZ)$, which, in addition to the properties listed above satisfies the following.
There exist a finite sequence of
tilings $\bt_0, \bt_1, \ldots, \bt_\ell = \bt_0$ in $\cT_{0}(\cR^{(k)})$
(for some $k \ge 0$)
and a sequence of integers $t_0, t_1, \ldots, t_{\ell}$ such that for every $i$, $\bt_i$ and $\bt_{i+1}$ are joined by a flip or trit,
$\Tw(\bt_i) = t_i \bmod m$, and $t_\ell = t_0 + m$.
\smallskip

One should stress that all known definitions of 
the twist of a tiling involve a significant combinatorial part
and are relatively complicated, cf. \cite{FKMS};
one of the main aims of the present paper
is to give an alternative geometric description.
%Notice that the proofs in the above references are rather delicate and quite different from anything we shall see here.
\end{remark}

%%%%%%%%%%%%%%%%%%%%%%%%%%%%%%%%%%%%%%%%%%%%%%%%%%%%%%%%%%%%%%%%%%%%%%%

\section{Rotation class and helicity of vector fields}
\label{section:rotation}

\subsection{Relative rotation class}
Let $M$ be a manifold with boundary $\partial M$ and volume form $\mu$. Consider a divergence-free vector field $\xi$
on $M$ and not necessarily tangent to its boundary. 

\begin{definition}{\rm (\cite{Arnold69})}
Consider the 2-form $\omega_\xi:=i_\xi\mu$.
It is closed, since $\xi$ is divergence-free, and hence it defines the 
cohomology class $[\omega_\xi]\in H^2(M)$. By the Lefschetz isomorphism
(the Poincar{\'e} isomorphism for manifolds with boundary)
$ H^2(M) \simeq H_1(M, \partial M)$.
The relative homology class defined by $[\omega_\xi]$ in $H_1(M, \partial M)$
is called a {\it relative rotation class} of the field $\xi$.

If $\partial M=\emptyset$ or $\xi$ is tangent to $\partial M$,
then $\omega_\xi|_{\partial M}=0$ and hence the corresponding
{\it rotation class} of the vector field $\xi$ 
is an element in $H_1(M) \simeq H^2(M, \partial M)$.
\end{definition}

This definition holds for $M$ of any dimension,
although our main application is in 3D.

%%%%%%%%%%%%%%%%%%%%%%%%%%%%%%%%%%%%%%%%%%%%%%%%%%%%%%%%%%%%%%%%%%%%%%%

\subsection{Helicity of fields and pipes}
\label{section:helipipe}

Now we assume that $M \subset {\RR}^3$ is a three-dimensional domain
and the field $\xi$ is divergence-free, 
tangent to $\partial M$, and has zero rotation class,
i.e. it is {\it exact} or {\it null-homologous}). 

\begin{definition}
\label{definition:Moffatt}
(\cite{Moffatt69})
The  {\it helicity}  of  a null-homologous  field $\xi$ 
in a domain $M \subset {\RR}^3$ is the number
$$
\Hel(\xi):=
\int \limits_{M} (\xi , {\rm{curl}}^{-1}\xi )\; d^3x,
$$
where the vector field
$\rm{curl}^{-1}\xi $ is a divergence-free vector potential of the field
$\xi$ (which exists since $\xi$ is exact), i.e.,   $\nabla \times (\rm{curl}^{-1}\xi)=\xi$ and 
$\rm{ div }(\rm{curl}^{-1}\xi)=0$.
\end{definition}

Let a divergence-free field $\xi$ be
confined to two narrow linked flux tubes.
Its helicity can be found explicitly as follows.
Suppose that the tube 
cores are closed curves $C_1$ and $C_2$, the  fluxes of the
field in the tubes are $\Flux_1$ and $\Flux_2$ (see Figure \ref{fig:link}).
The curves $C_i$ are oriented so that $\Flux_i =Q_i\ge 0$.
Assume also that there is no net
twist within each tube or, more precisely, that
the field trajectories foliate each of the
tubes into  pairwise unlinked circles and the periods of those trajectories are equal.
One can show 
that the helicity invariant of such a field is
given by
$$
\Hel(\xi ) = 2\,{\lk}(C_1,C_2)\cdot \Flux_1\cdot \Flux_2, 
$$
where
${\lk}(C_1,C_2)$ is the (Gauss) {\it linking number} of $C_1$ and $C_2$,
which explains the term ``helicity'' coined in \cite{Moffatt69},
as the measure of coiling one curve about the other.
Recall, that the  number ${\lk}(C_1,C_2)$ for two oriented closed curves
is the signed number of the intersection points
of one curve with an arbitrary oriented surface
spanning the other curve.

\begin{figure}[ht]
\begin{center}
\includegraphics[scale=0.7]{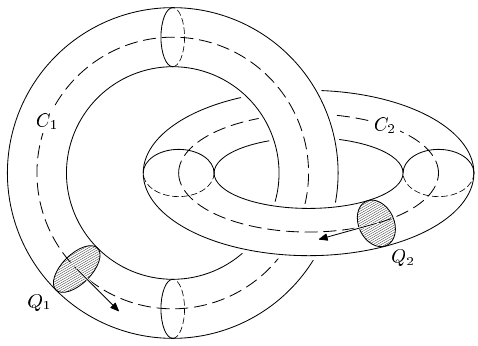}
\end{center}
\caption{Two linked tubes:
the solid tori are thin tubular neighborhoods of the curves $C_i$.
A few cross-sections are shown, while
the vector field is  transversal to these sections.}
\label{fig:link}
\end{figure}

Suppose now that a divergence-free field $\xi$ is 
confined to several narrow linked oriented tubes
with cores given by closed curves $C_i$ and  fluxes $\Flux_i \ge 0$.
We still assume that  the field trajectories foliate each of the
tubes into topological circles.
However we do not assume any longer that these field trajectories (circles)
inside the same tubes are pairwise unlinked, while their periods can now be arbitrary.
The corresponding linking of circles inside the $i$th tube
(given by the linking number of the core curve $C_i$
with any satellite  curve)
is called the self-linking ${\slk}(C_i)$
of the curve $C_i$ in the corresponding pipe.

\begin{prop}
\label{prop:Hel}
The   helicity  of such a field $\xi$ is
\begin{equation}
\label{equation:Hel}
\Hel(\xi ) =
2\sum_{i<j} {\lk}(C_i,C_j)\cdot \Flux_i\cdot\Flux_j
+ \sum_i {\slk}(C_i) (\Flux_i)^2.
\end{equation}
\end{prop}

We recall the proof of this folklore formula in Appendix, and will use it for  computations related to tilings.
Note that self-linking cannot be defined for an isolated closed curve
(since there is no canonical choice of a satellite curve),
but only for a curve with a framing, which delivers such a choice.
For a field trajectory in a pipe,
neighbouring trajectories provide the framing.

\medskip

\begin{remark}
While helicity was defined above by using the Riemannian metric on $M$, 
it is actually a  topological characteristic of a divergence-free vector field,
depending only on the choice of a volume form on the manifold. Namely, 
consider a manifold  $M$ (possibly with boundary) 
with a volume form $\mu$, and let $\xi$ be a null-homologous vector 
field on $M$ (tangent to the boundary). 
The divergence-free condition means that 
the Lie derivative of $\mu$ along $\xi$ vanishes: $L_{\xi}\mu=0$, or, which 
is the same, the substitution  $i_\xi\mu=:\omega_\xi$ 
of the field $\xi$ into the 3-form $\mu$ is a closed 2-form:
$d\omega_\xi=0$. If $\xi$ is moreover null-homologous, then $\omega_\xi$ is actually an exact 2-form: 
$\omega_\xi=d\alpha$ for some  1-form $\alpha$, called a potential.
(On a simply connected manifold $M$ any divergence-free vector field is null-homologous.)
\end{remark}

\begin{definition} {\rm (\cite{Arnold73})}
\label{definition:Arnold}
The {\it  helicity }  $ \Hel(\xi)$ of a
null-homologous field $\xi$ on a   three-dimensional manifold $M$
equipped with a
volume element $\mu$ is the integral of the wedge product of the
form $\omega_\xi:=i_\xi\mu$ and   its  potential:
$$
\Hel(\xi)
= \int_M d\alpha\wedge \alpha, {~{\rm where}~} d\alpha =
\omega_\xi.
$$
As discussed, this generalizes Definition~\ref{definition:Moffatt}.
\end{definition}
 
An immediate consequence of this purely topological (i.e. metric-free) 
definition is the following Arnold's theorem: 
The helicity $\Hel(\xi)$ is preserved under the action on $\xi$
of a volume-preserving diffeomorphism of $M$. 
In this sense $\Hel(\xi)$ is a topological invariant: it can be defined 
without coordinates or a choice of metric, and hence 
every volume-preserving diffeomorphism carries
a field $\xi$ into a field with the same helicity. The 
physical significance of helicity is  due to the fact that
it appears  as a conservation law in both fluid mechanics and magnetohydrodynamics: 
Kelvin's law implies the invariance of helicity of the vorticity field for an ideal fluid motion.
\bigskip
%%%%%%%%%%%%%%%%%%%%%%%%%%%%%%%%%%%%%%%%%%%%%%%%%%%%%%%%%%%%%%%%%%%%%%%

\subsection{Relative helicity for vector fields}
\label{section:relhel}

First we recall the definition of relative helicity  for vector fields
(elaborating the definitions in \cite{BF} and \cite{AK}).
Suppose that a domain in the space $\RR^3$ 
(or a closed oriented manifold $M^3$) 
is split into two regions $A$ and $B$
separated by a boundary surface $S$.
Assume further that  two divergence-free vector
fields $\xi$ and $ \eta$   in $A$ coincide on the boundary $S$ and have
the same extension $\zeta$ into the region $B$. 
Call the extended fields in $M$  respectively 
$\tilde \xi$ and $\tilde \eta$.
Abusing notation we will 
denote them as the sums
$\tilde \xi=\xi+\zeta$ and $\tilde \eta=\eta+\zeta$, where
$\xi$, $\eta$ and $\zeta$ are
regarded as the (discontinuous) vector fields in the entire manifold $M$
with supports $\supp\xi,\supp\eta\subseteq A$ and 
$\supp\zeta\subseteq B$.
(Alternatively one can modify the fields $\xi$, $\eta$, $\zeta$
in a narrow neighborhood of the boundary $S$ to avoid discontinuity;
this will not affect the argument below.)

Assume that both the extended fields $\tilde \xi$ and $\tilde \eta$
are null-homologous in $M$ and tangent to $\partial M$,
while we do not impose any restrictions on the topology
of the sets $A$, $B$, or $M$.

\bigbreak

\begin{definition}
\label{definition:RHel}
The difference $\Delta\Hel(\xi,\eta) := \Hel( \tilde \xi)- \Hel(\tilde  \eta)$ 
measures the {\em relative helicity} of the fields $\xi$ and $\eta$ in $A$. 
\end{definition}

% {\bf Lemma 2.1.}  
\begin{theo}
\label{theo:AB}
The relative helicity
$\Delta\Hel(\xi,\eta) = \Hel( \tilde \xi)- \Hel(\tilde  \eta)$
% of the  fields $\tilde \xi$ and $\tilde \eta$
is independent of their common extension $\zeta$ in the region $B$. 
\end{theo}

\begin{proof}
Define the (closed) two-forms $\alpha, \beta,$ and $\omega$ 
by substituting   the vector fields $\xi,\eta,$ and $\zeta$ 
into  the volume form $\mu$ on $M$:
$i_\xi\mu=\alpha$, $i_\eta\mu=\beta$, and $i_\zeta\mu=\omega$.
%(or equivalently, $\xi=curl~\alpha$), 
Then one has to show that the difference
\[
\Delta\Hel = \Hel(\tilde\xi)-\Hel(\tilde\eta)
=\int_M (\alpha+\omega)\wedge 
d^{-1}(\alpha+\omega)
-\int_M (\beta+\omega)\wedge d^{-1}(\beta+\omega)
\]
does not depend on $\omega$.
Note that the 2-forms $\alpha+\omega$ and $\beta+\omega$ are exact in $M$
(because of the assumption that  the fields
$\tilde \xi$ and $\tilde \eta$ are null-homologous)
and taking 
$d^{-1}$ of them makes sense. 

Assume first that the regions $A$ and $B$ are both simply-connected.
One readily obtains
\[
\Delta\Hel =\int_M \alpha\wedge d^{-1}\alpha-\int_M \beta\wedge d^{-1}\beta 
+\int_M (\alpha-\beta)\wedge d^{-1}\omega
+\int_M \omega\wedge d^{-1}(\alpha-\beta).
\]
Here $d^{-1}$ applied to  a discontinuous 2-form is a continuous 
 1-form (the ``form-potential").
The terms in $\Delta\Hel $ containing 
$\omega$  
are $\int_M (\alpha-\beta)\wedge d^{-1}\omega
+\int_M \omega\wedge d^{-1}(\alpha-\beta)$,
and we want to show that their contribution vanishes. 

Integrating by parts
one of the terms, we  come  to 
$2\int_M (\alpha-\beta)\wedge d^{-1}\omega $, which,   in turn,
 is equal to  
$2\int_A (\alpha-\beta)\wedge d^{-1}\omega $, 
since $\supp(\alpha-\beta)\subseteq A$.

On the other hand, in the domain $A$ the 1-form  $d^{-1}\omega $ is the 
differential  of a function, 
$d^{-1}\omega =dh$.
Indeed,  it is closed
(the differential $d(d^{-1}\omega)=\omega$ vanishes in $A$
due to the condition 
on $\supp\zeta=\supp\omega\subseteq B$),
and hence it is exact in the  simply connected region $A$.
Hence,
\[
2\int_A (\alpha-\beta)\wedge d^{-1}\omega=2\int_A (\alpha-\beta)\wedge dh=
2\int_S h(\alpha-\beta)=0,
\]
where the last equality is due to the assumption on  the  
identity of the fields 
$\xi$ and $\eta$ on the boundary $S$.  
This proves that $\Delta\Hel $ is unaffected by the choice of 
the extension $\zeta$. 

Now consider the case of arbitrary domains $A$ and $B$.
In the latter case, the 2-form $\omega$ is not exact, but only closed
(and such that $\alpha+\omega$ and $\beta+\omega$ are exact,
i.e. $[\alpha+\omega]=[\beta+\omega]=0\in  H^2(M)$).
To emphasize the ambiguity in the choice of $\omega$ we represent is as   $\omega=\omega_0+\gamma$,
where we fix some `reference' closed 2-form $\omega_0$, while an arbitrary 2-form $\gamma$ is exact,
and both  $\omega_0$ and $\gamma$ have their support in $B$.
Now the difference we are studying is
\[
\Hel(\tilde\xi)-\Hel(\tilde\eta) =\int_M (\alpha+\omega_0+\gamma)\wedge 
d^{-1}(\alpha+\omega_0+\gamma)
-\int_M (\beta+\omega_0+\gamma)\wedge d^{-1}(\beta+\omega_0+\gamma)\,,
\]
and we would like to show
that it is independent of the choice of an exact 2-form $\gamma$.
We claim that this is evident 
once we introduce new forms
$\bar\alpha=\alpha+\omega_0$ and $\bar\beta=\beta+\omega_0$
and rewrite the difference as 
\[ \Hel ( \tilde \xi)-\Hel (\tilde  \eta)
=\int_M (\bar\alpha+\gamma)\wedge 
d^{-1}(\bar\alpha+\gamma)
-\int_M (\bar\beta+\gamma)\wedge d^{-1}(\bar\beta+\gamma)\,, \]
thus mimicking the expression above for simply-connected regions.  
Indeed, now it boils down to the same computation leading to the form
\[
2\int_A (\bar\alpha-\bar\beta)\wedge d^{-1}\gamma=
2\int_A (\alpha-\beta)\wedge d^{-1}\gamma=0\,.
\]
The latter expression vanishes for the same reason as above:
$\alpha-\beta$ has the support in $A$ and it is exact
(since both $\alpha+\omega$ and $\beta+\omega$ are exact),
while $d^{-1}\gamma$ is closed in $A$.
This concludes the proof that relative helicity
is well-defined in the general case.
\end{proof}

% {\bf Definition 2.2.} 
% {\it The difference $\Delta  Hel:= Hel( \tilde \xi)- Hel(\tilde  \eta)$ 
% measures the {\rm relative helicity} of the fields $\xi$ and $\eta$ in $A$. }
% \medskip
% We take note of a formula obtained in the above proof.
% If $\alpha$ and $\beta$ are $2$-forms, as above, then
% \[
% \Delta\Hel =
% \int_M \alpha\wedge d^{-1}\alpha-\int_M \beta\wedge d^{-1}\beta \,.
% \]

%%%%%%%%%%%%%%%%%%%%%%%%%%%%%%%%%%%%%%%%%%%%%%%%%%%%%%%%%%%%%%%%%%%%%%%

\section{Pipes, fluxes, and shells}
\label{sect:all-pipes}
In this section we present the constructions
taking us from the context of cubiculated regions and domino tilings
to that of vector fields and helicities.

\subsection{The five pipes construction}
\label{section:fivepipes}

In the first construction,
for a domino tile $d \subset \RR^3$,
we construct a specific divergence-free vector field $\xi_d$ in $d$.
The vector field $\xi_d$ is confined to five narrow flux tubes.
The tubes are the boundaries of tubular neighborhoods
of smooth curves approximating the polygonal lines.
This is therefore the situation
discussed in Section \ref{section:helipipe}.
If the domino is $d = [0,2] \times [0,1] \times [0,1]$,
the tubes are thin neighborhoods
of the following five oriented polygonal lines,
also shown in Figure~\ref{fig:5pipe}:
\begin{gather*}
\textstyle
(0,\frac12,\frac12)\longrightarrow(2,\frac12,\frac12), \\
\textstyle
(\frac12,0,\frac12)\to(\frac12,\frac14,\frac12)\to
(\frac32,\frac14,\frac12)\to(\frac32,0,\frac12), \\
\textstyle
(\frac12,\frac12,0)\to(\frac12,\frac12,\frac14)\to
(\frac32,\frac12,\frac14)\to(\frac32,\frac12,0), \\
\textstyle
(\frac12,1,\frac12)\to(\frac12,\frac34,\frac12)\to
(\frac32,\frac34,\frac12)\to(\frac32,1,\frac12), \\
\textstyle
(\frac12,\frac12,1)\to(\frac12,\frac12,\frac34)\to
(\frac32,\frac12,\frac34)\to(\frac32,\frac12,1). 
\end{gather*}
Notice that the figure is symmetric under rotations
 by $\frac{\pi}{2}$ around the central line,
and also under reflection on the plane $y=\frac12$.
The narrow tubes are assumed to be likewise symmetric.

\begin{figure}[ht]
\begin{center}
\includegraphics[scale=0.17]{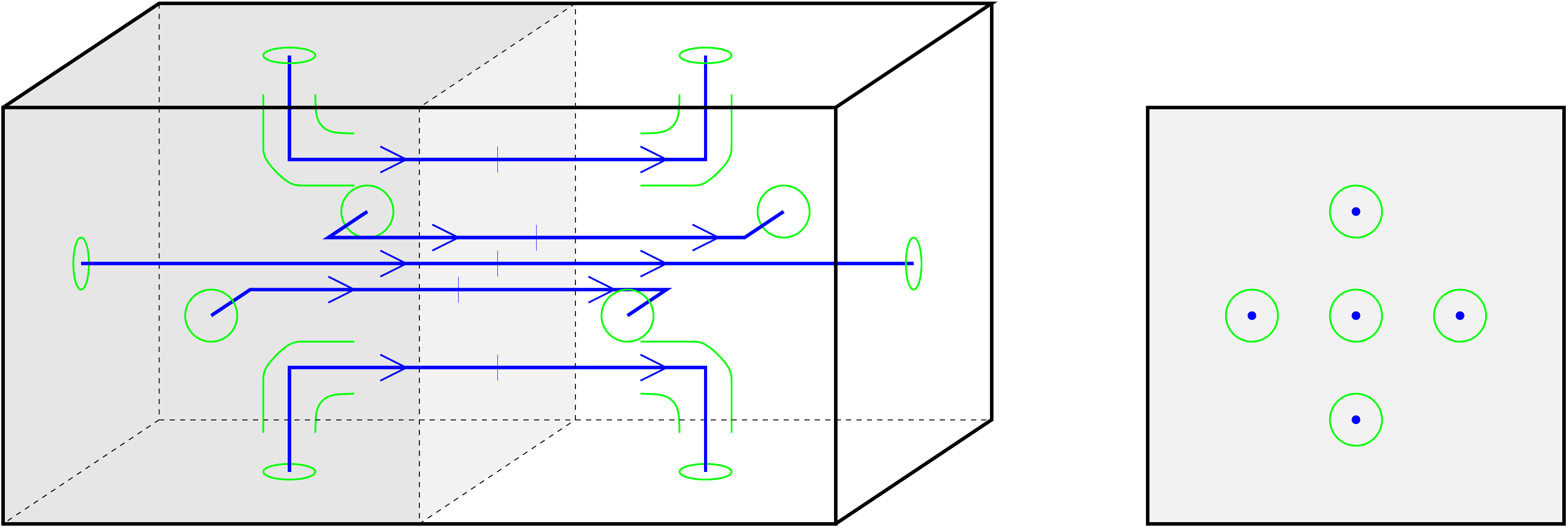}
\end{center}
\caption{A domino, five polygonal lines inside it
and the respective smooth tubes approximating the lines.
Tubes are only very partially drawn in order to keep the figure simple.
Lines are always oriented from black to white.
At the right, the central square separating the two unit cubes of the domino,
crossed by five lines and five tubes.}
\label{fig:5pipe}
\end{figure}

For other dominoes, the figure is appropriately rotated and translated,
while the arrows always point from a black cube to a white one.
Thus, pipes enter the domino through a neighborhood of the center
of a face of the black cube and
exit the domino near the center of a face of the white cube.

If we have a tiling $\bt$ of a region $\cR$,
the same construction is performed in each domino.
This produces a vector field $\xi_{\bt}$ in $\cR$.
Assume that the flux through each pipe is equal to the same $\varphi > 0$.
We may assume that $\xi_{\bt}$
is smooth in the interior of $\cR$ with the support in the union of the pipes.
Notice that the restriction of $\xi_{\bt}$
to a thin neighborhood of the boundary $\partial\cR$
does not depend on the choice of the tiling $\bt$.

%%%%%%%%%%%%%%%%%%%

\subsection{Proof of theorem on relative flux}
\label{sect:rotation}

Here we prove Theorem \ref{theo:main2}
that the relative flux $\RFlux(\bt)\in H_1(\cR, \partial \cR)$ 
of a tiling $\bt$ coincides modulo a factor with 
the relative rotation class $[\xi_\bt]$ the vector field $\xi_\bt$ 
obtained via the 5-pipe construction:
$[\xi_\bt]=6\varphi \RFlux(\bt)$.

\begin{proof}[Proof of Theorem~\ref{theo:main2}]
Adjust the field $\xi_\bt$  as follows.
Add to it one more pipe in the shape of a long torus,
see Figure~\ref{fig:longT}, 
which will follow the black-to-white direction as on Figure~\ref{fig:5pipe}
and then returns after a U-turn,
without self-linking or linking with any other pipe.  

\begin{figure}[ht]
\begin{center}
\includegraphics[scale=0.17]{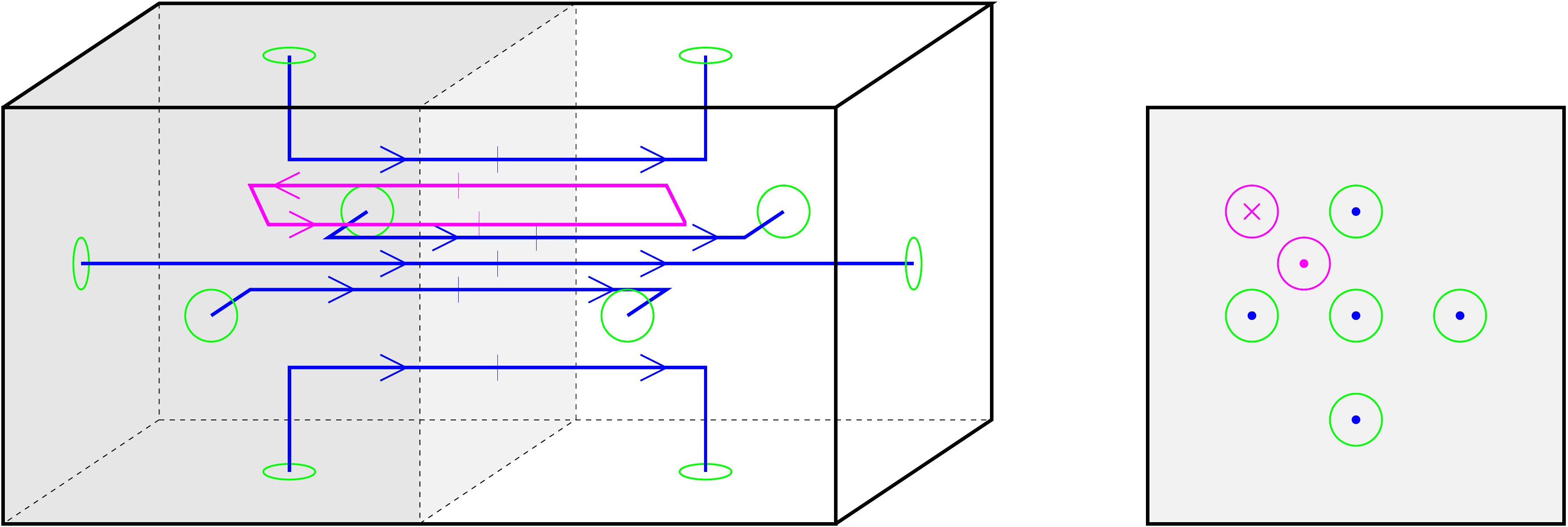}
\end{center}
\caption{An additional narrow toroidal pipe
with a field flux $\varphi$ and directed along the axis of the tile.}
\label{fig:longT}
\end{figure}

With this addition the new field $\bar \xi_\bt$ in the tile:
(1) is still divergence-free without singular points,
(2) has 6 pipes in the black-to-white direction
and 1 tube in the opposite direction,
(3) has one pipe from each face towards the vicinity
of the center of the black cube and one pipe
from (the vicinity of) the center of the white cube toward its faces,
(4) the flux in each narrow pipe is $\varphi$.  

With this construction, mimicking the definition of the chain $6\bt - \bq_1$,
one achieves that  the intersection number 
of the latter (multiplied by $\varphi$)
with any closed  2D surface in $\cR$  transversal to this chain
is equal  to the flux of the corresponding vector field
$\bar \xi_\bt$ through such a surface.
Hence $[\bar \xi_\bt]$ realizes
the same homology class in $H_1(\cR, \partial \cR)$
as $\varphi[6\bt - \bq_1]$.
\end{proof}

\begin{remark}
\label{remark:xiRFlux}
While this adjusted ``6-pipe" field $\bar \xi_\bt$
more adequately represents the $6\bt - \bq_1$,
the contribution of the extra ``long solid torus" pipe vanishes
to both the flux and the helicity, discussed below,
so one can confine to the 5-pipe construction.
However, it explains the origin of the factor $6$ in the formulas.
\end{remark}

\begin{remark}
The chain $\bq_1$, as well as the part of the field $ \bar \xi_\bt$ defined by it, is determined solely by the coloring of the region $\cR$. It defines a net of pipes with sources at each white center and sinks at each black center, ``following" the chain 
$\bq$. Then the tiling $\bt$ can be regarded as an additional system of pipes, 6 times stronger, joining the neighboring centers according to the tiling geometry and satisfying the Kirchhoff  junction rule. 

% Finally,
Note that the same claim about proportionality
of the relative flux $\RFlux(\bt)\in H_1(\cR, \partial \cR)$ 
and the relative rotation class $[\xi_\bt]$ of the corresponding field
holds mutatis mutandis for tilings in any dimension.
\end{remark}

%%%%%%%%%%%%%%%%%

\subsection{An isolating shell}
\label{sect:isolate}

Consider first the case where $\cR \subset \RR^3$
is a well behaved region, perhaps a box.
We apply the construction discussed in Section~\ref{section:relhel}.
Let $M = \RR^3$, $A = \cR$ (possibly rounded at the corners)
and $B = M \smallsetminus A$.
Choose a divergence-free vector field $\zeta$ in $B$
coinciding with $\xi_{\bt}$ near $\partial\cR$.
For simplicity, we may choose $\zeta$ to be also confined to a few tubes.
We call $\zeta$ or, more precisely, $(M,A,B,\zeta)$,
an \textit{isolating shell}.
Given a tiling $\bt$, let $\tilde\xi_{\bt}$
be the corresponding smooth, divergence-free vector field in $M$.
Since $M$ is contractible, the extended vector field $\tilde\xi_{\bt}$
is trivially null-homologous.

The more general  case is similar, with a few adjustments.
If $\partial\cR=\emptyset$, nothing needs to be done.
Otherwise, recall the assumption that $\partial\cR$ is a topological manifold.
We can therefore construct an open manifold $M_0$
with $A = \cR \subset M_0$ by taking $M_0$ to be the union of $A$
with a tubular neighborhood of $\partial\cR$.
A minor difficulty is that $M_0 \smallsetminus A$ may well be disconnected.
If this happens, we construct $M \supseteq M_0$
by adding tubes (disjoint from $\cR$)
connecting the  components of $M_0 \smallsetminus A$.
We thus have an open manifold $M$, $A = \cR \subset M$
and a connected subset $B = M \smallsetminus A$.
As above, the vector fields $\xi_{\bt}$ for all tilings $\bt$
coincide in a neighborhood
of $\partial\cR \subset M$.
Again, choose a  divergence-free vector field $\zeta$ in $B$,
also confined to a few tubes and
coinciding with $\xi_{\bt}$ near $\partial\cR$.
The desired isolating shell is  $(M,A,B,\zeta)$.
We need to check whether the extended vector field
$\tilde\xi_{\bt}=\xi_{\bt}+\zeta$ is null-homologous in $M$.

\begin{lemma}
\label{lemma:RFlux}
Consider a cubiculated region $\cR$ and a domino tiling $\bt$.
For $M \supset \cR$ as above, the vector field $\zeta$ can be chosen
so that $\tilde\xi_{\bt}=\xi_{\bt}+\zeta$ is null-homologous in $M$
if and only if $\RFlux(\bt) = 0$.
\end{lemma}

\begin{proof}
In this proof, we interpret the vector fields as 
thin tubular neighborhoods of weighted oriented curves.
Thus, we can 
obtain a class in $H_1(M)$ (with real coefficients)  from a vector field.

Assume first that $\zeta$ has been chosen so that
$\tilde\xi_{\bt}$ is null-homologous in $M$.
Inclusion defines a map $i: H_1(M) \to H_1(M,B) = H_1(\cR,\partial\cR)$.
Since $[\tilde\xi_{\bt}] = 0$ we have
$i([\tilde\xi_{\bt}]) = 0 \in H_1(\cR,\partial\cR)$.
But $i([\tilde\xi_{\bt}]) = [\xi_{\bt}] = 6\varphi \RFlux(\bt)$
(from Theorem~\ref{theo:main2}),
proving one implication.

Conversely, assume that $\RFlux(\bt) = 0$.
Construct a weighted oriented surface $S$ in $\cR$
such that $\partial S$ coincides with $\bt - \frac{\bq_1}{6}$
in the interior of $\cR$.
Use the part of $\partial S$ on $\partial\cR$ as guides
to construct the required pipes in the intersection
of $B$ with a thin tubular neighborhood of $\partial\cR$,
i.e., to construct $\zeta$.
The surface $S$ is a witness to the fact that 
$\tilde\xi_{\bt}$ is null-homologous in $M$.
\end{proof}

%%%%%%%%%%%%%%%%%%%%%%%%%%%%%%%%%%%%%%%%%%%%%%%%%%%%%%%%%%%%%%%%%%%%%%%

\section{Key examples}
\label{section:firstexamples}

In this section we present the first examples
of isolating shells and of the five pipes construction.
The results of these examples will be used in later proofs.

\begin{example}
\label{example:221}
Let $\cR = [0,2]\times[0,2]\times[0,1]$,
the $2\times2\times1$ box:
$\cR$ admits exactly two tilings, $\bt_0$ and $\bt_1$,
shown in Figure~\ref{fig:221}.

\begin{figure}[ht]
\begin{center}
\includegraphics[scale=0.5]{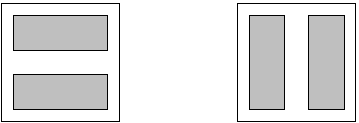}
\end{center}
\caption{The only two tilings of the $2\times2\times1$ box
are joined by a flip.}
\label{fig:221}
\end{figure}

\begin{figure}[ht]
\begin{center}
\includegraphics[scale=0.3]{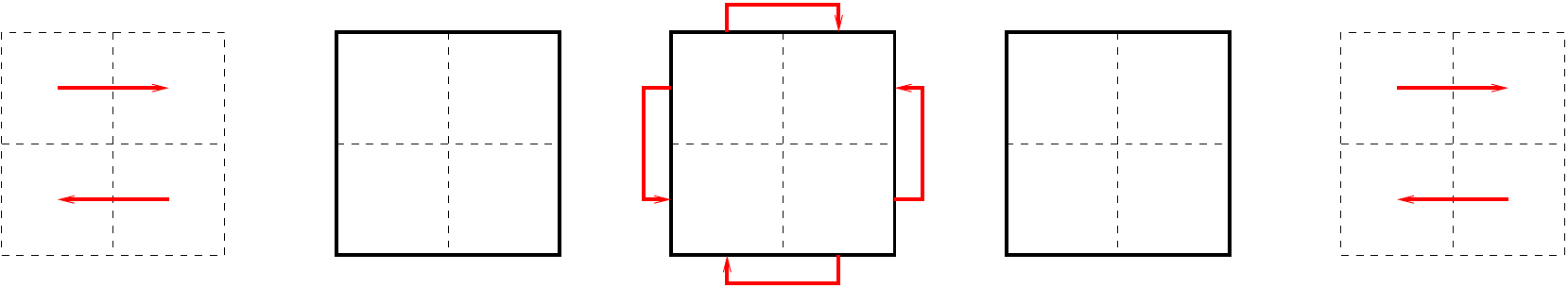}
\end{center}
\caption{A valid isolating shell for the $2\times2\times1$ box.
We show the planes $z = -\frac14, \frac14, \frac12, \frac34, \frac54$.}
\label{fig:flip5c1}
\end{figure}

Figure~\ref{fig:flip5c1} shows a valid isolating shell for $\cR$.
Figure~\ref{fig:flip5c2} shows the two vector fields
$\tilde\xi_{\bt_0}$ and $\tilde\xi_{\bt_1}$:
both have zero helicity.
Indeed, a curve $C_0$ is \textit{trivial}
(for a family of curves $(C_i)$ including $C_0$)
if there exists a contractible open neighborhood $A \supset C_0$
intersecting no other curve $C_i$
% $C_0$ is unknotted
and ${\slk}(C_0) = 0$.
In this case, $C_0$ can be discarded from the family
without changing the helicity.
In the present example, every curve is trivial.
\end{example}

\begin{figure}[ht!]
\begin{center}
\includegraphics[scale=0.3]{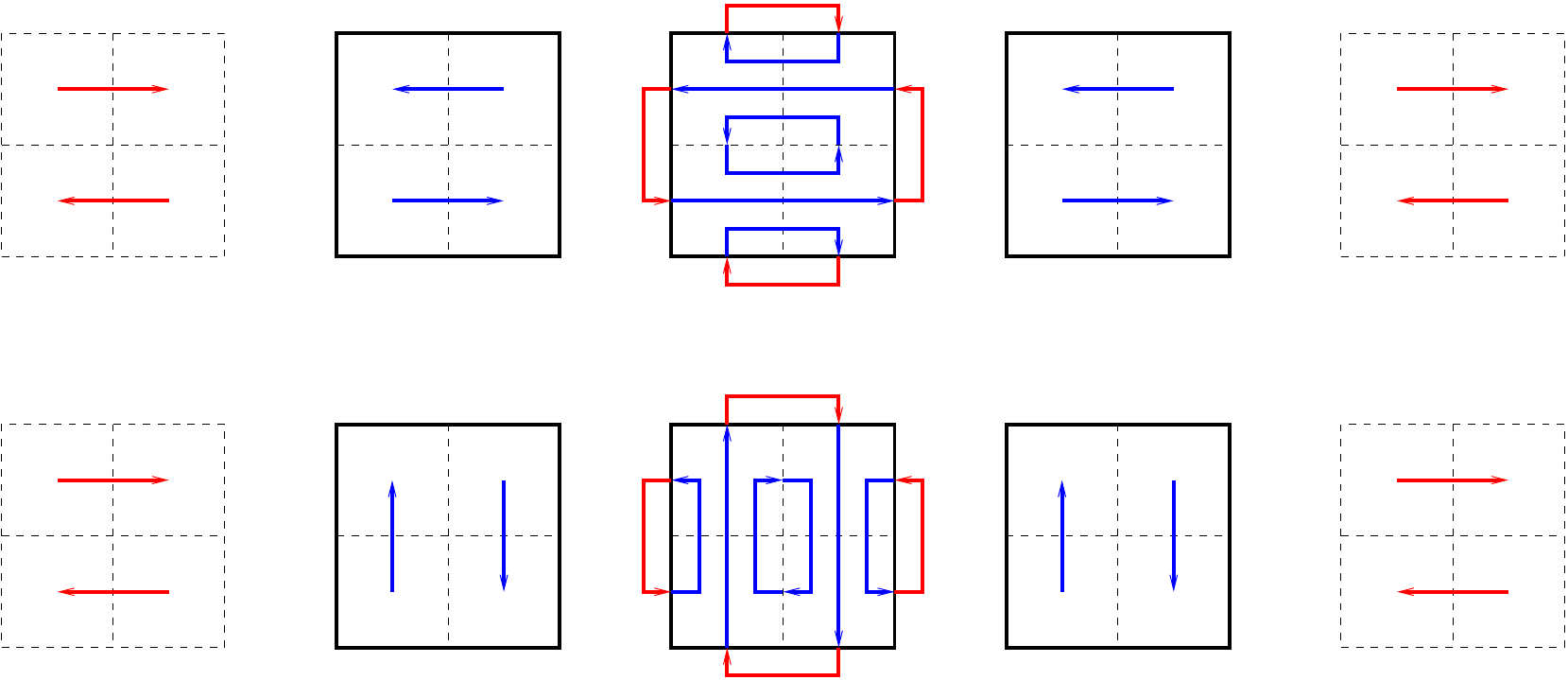}
\end{center}
\caption{The vector fields $\tilde\xi_{\bt_i}$
for the only two tilings of the $2\times2\times1$ box
(with the same planes as Figure~\ref{fig:flip5c1}).}
\label{fig:flip5c2}
\end{figure}

\bigbreak

\begin{example}
\label{example:hex}
Let $\cR$ be the region shown in Figure~\ref{fig:hex0}, i.e.,
\[ \cR = [0,2]^3 \smallsetminus
( ( (1,2] \times (1,2] \times [0,1) ) \cup 
( [0,1) \times [0,1) \times (1,2] ) ). \]
It is a $2\times 2\times 2$ cube without two small, $1\times 1\times 1$ cubes on a diagonal.
The region $\cR$ admits precisely two tilings $\bt_0$ and $\bt_1$,
also shown in Figure~\ref{fig:hex0}.
Notice that the two tilings differ by a positive trit
from $\bt_0$ to $\bt_1$.
They also differ by reflection on the plane $x=y$.
In this case there is no natural definition of $\Tw(\bt)$
as an integer, that is,
we need to choose a base tiling in a more or less arbitrary manner.
If we choose $\bt_0$ as a base tiling we then have
$\Tw(\bt_0) = 0$ and $\Tw(\bt_1) = 1$.

\begin{figure}[ht]
\begin{center}
\includegraphics[scale=0.5]{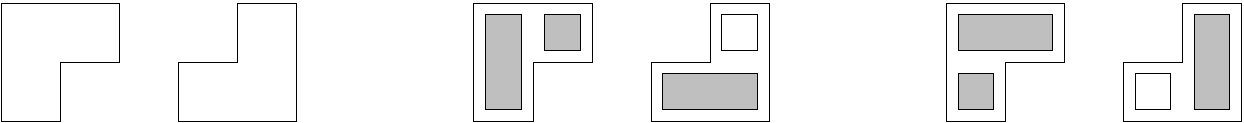}
\end{center}
\caption{The region $\cR$ and its two tilings $\bt_0$ and $\bt_1$.}
% discussed in Example~\ref{example:hex}
\label{fig:hex0}
\end{figure}

\begin{figure}[ht]
\begin{center}
\includegraphics[scale=0.24]{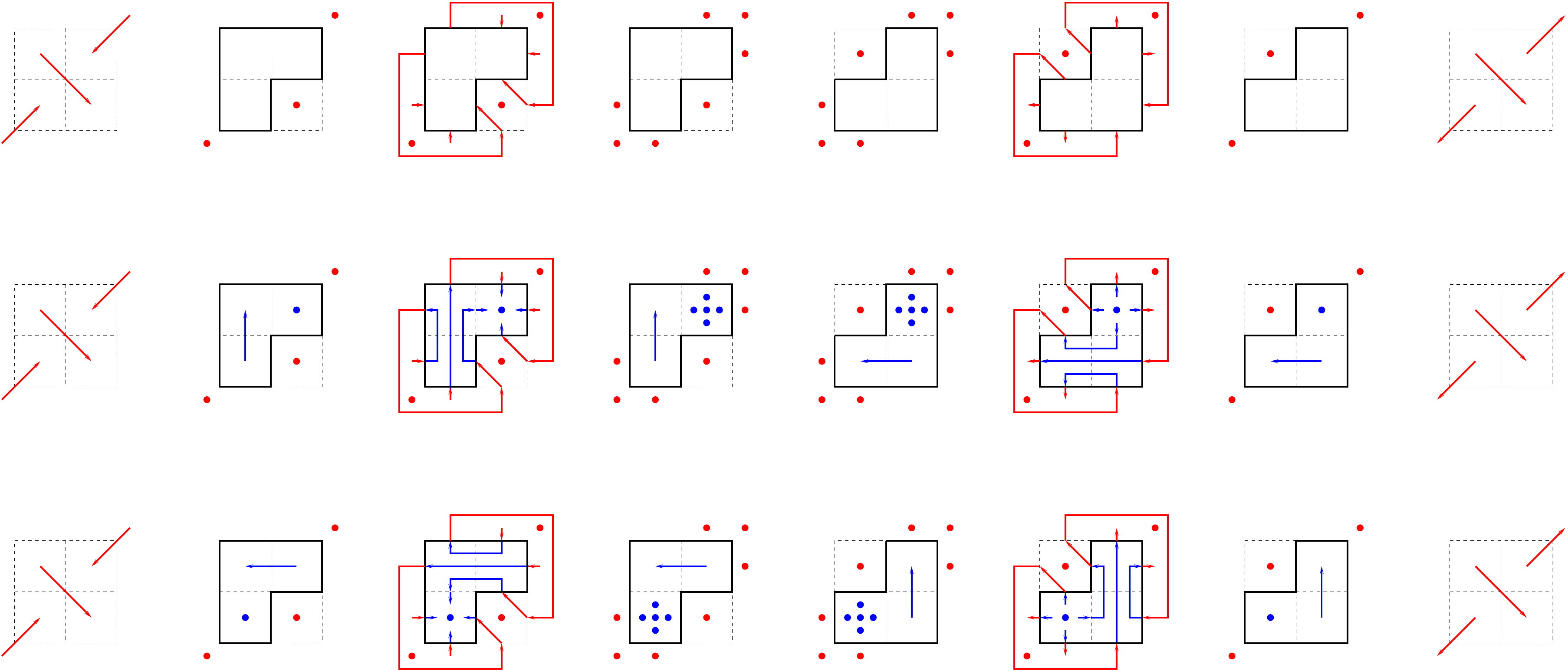}
\end{center}
\caption{A valid isolating shell
and the vector fields $\tilde\xi_{\bt_0}$ and $\tilde\xi_{\bt_1}$.
We show the following planes:
$-\frac14$, $\frac14$, $\frac12$, $\frac34$,
$\frac54$, $\frac32$, $\frac74$, $\frac94$.}
% discussed in Example~\ref{example:hex}
\label{fig:hex1}
\end{figure}

Figure~\ref{fig:hex1} shows a valid isolating shell for $\cR$
and the two vector fields $\tilde\xi_{\bt_0}$ and $\tilde\xi_{\bt_1}$.
Notice that $\tilde\xi_{\bt_1}$ is the mirror image of $\tilde\xi_{\bt_0}$
under reflection on the plane $x=y$.
This implies $\Hel(\tilde\xi_{\bt_1}) = - \Hel(\tilde\xi_{\bt_0})$.

\begin{figure}[ht]
\begin{center}
\includegraphics[scale=0.18]{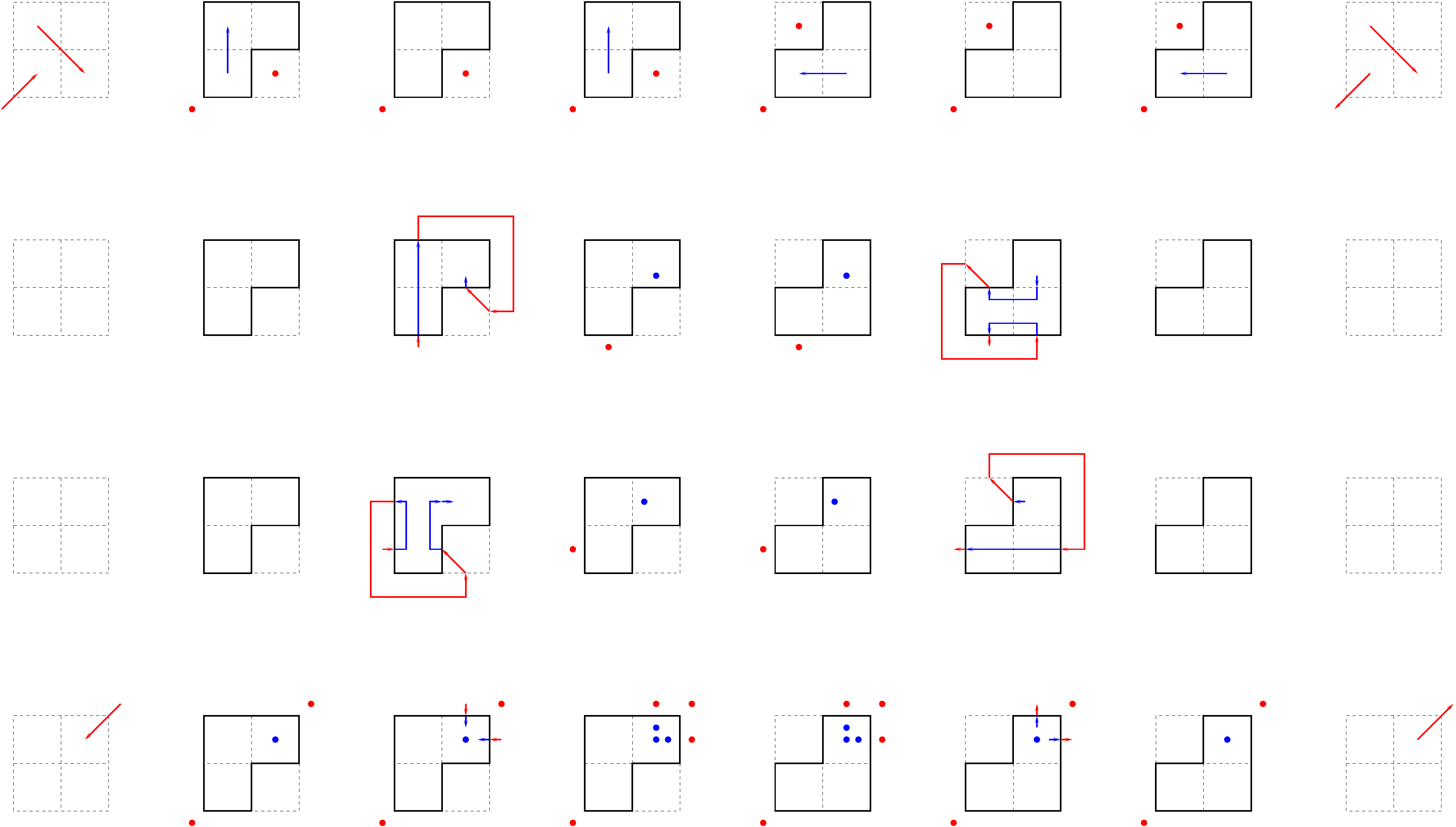}
\end{center}
\caption{Curves for the vector fields $\tilde\xi_{\bt_0}$
from Example~\ref{example:hex}.
The first three rows show three nontrivial curves;
the last row shows several trivial curves.}
% discussed in Example~\ref{example:hex}
\label{fig:hex3}
\end{figure}

\begin{figure}[ht]
\begin{center}
\includegraphics[scale=0.2]{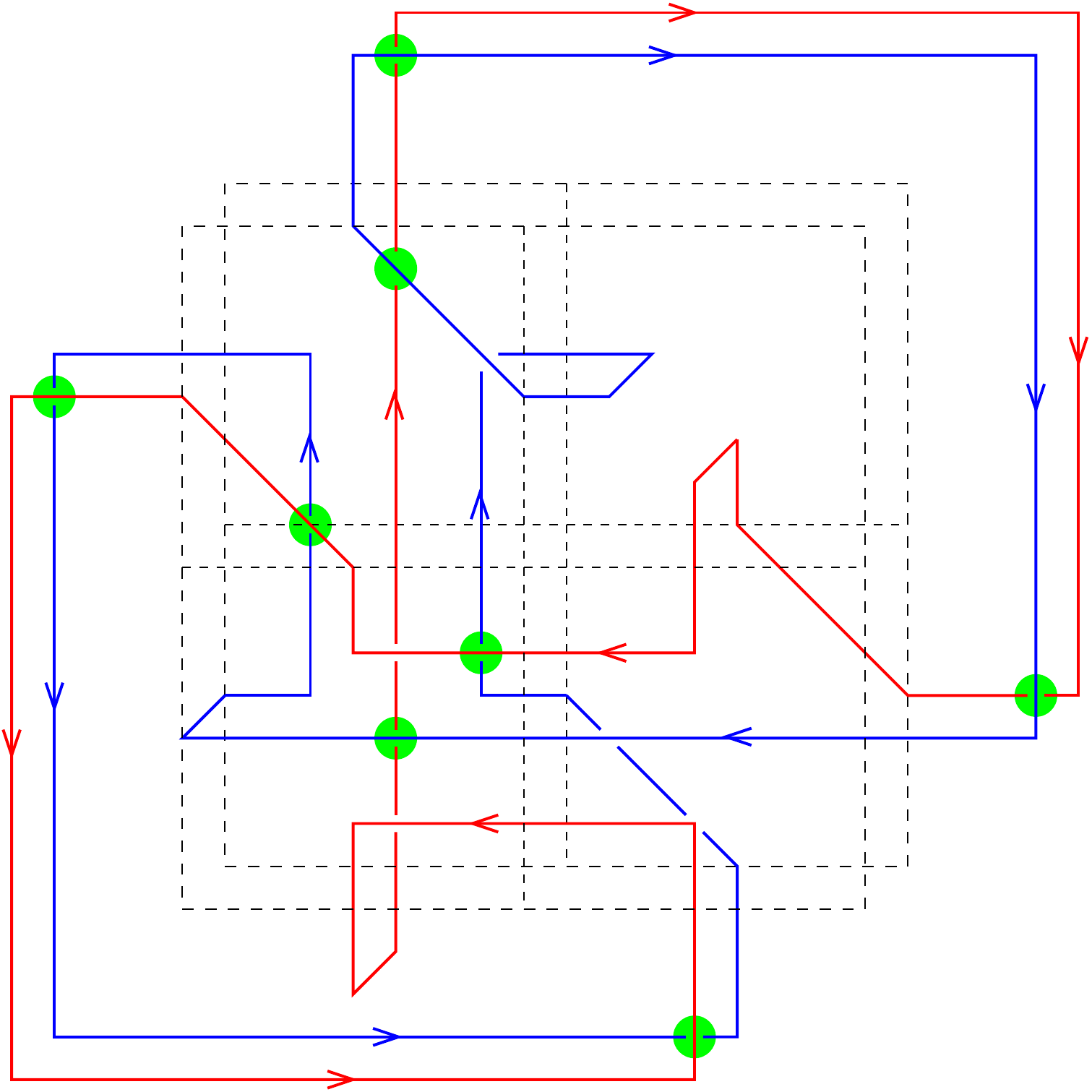}
\end{center}
\caption{Example of computation of the linking number,
between the curves in the second and third rows of Figure~\ref{fig:hex3};
see Example~\ref{example:hex}.}
\label{fig:hex4}
\end{figure}

\begin{figure}[ht]
\begin{center}
\includegraphics[scale=0.2]{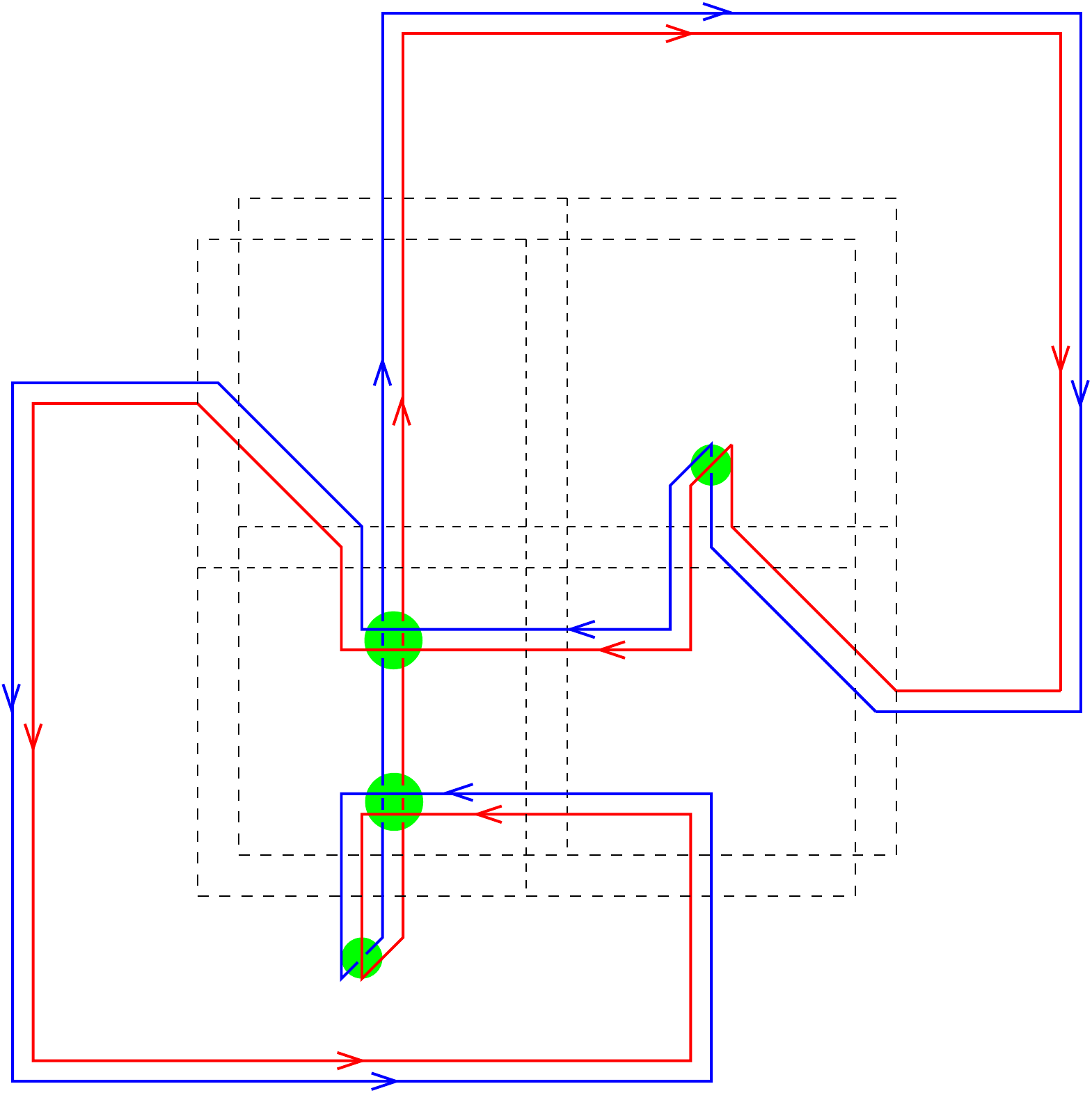}
\end{center}
\caption{Example of computation of the self-linking number
for the curve in the second row of Figure~\ref{fig:hex3};
see Example~\ref{example:hex}.}
\label{fig:hex5}
\end{figure}

Figure~\ref{fig:hex3} shows the several curves
describing $\tilde\xi_{\bt_0}$.
There are three nontrivial curves $C_i$, $1 \le i \le 3$
(one per row in Figure~\ref{fig:hex3})
and several trivial curves.
An explicit computation shows that 
${\lk}(C_i,C_j)=-2$ for all $i\not=j$ and ${\slk}(C_i)=-2$ for all $i$.
Figure~\ref{fig:hex4} shows a sample computation that ${\lk}(C_2,C_3)= -2$.
Indeed, we draw projections of these two curves onto (a small perturbation of) the horizontal $xy$ plane
(red for $C_2$, blue for $C_3$).
This perturbed plane is chosen so that the curves' projections 
on it intersect transversally.
There are eight intersection points
between the projections of different curves (indicated in green).
Computing the signs, we verify that six of them are negative
and two are positive, yielding ${\lk}(C_2,C_3)= -2$.
Similarly, Figure~\ref{fig:hex5} shows that ${\slk}(C_2) = -2$.
Here we project $C_2$ (red) and a satellite curve $C_2^{\sat}$ (blue).
There are six intersection points between different curves
(indicated in green,
with larger disks indication pairs of twin intersection points).
There are five negative and one positive intersection
and therefore ${\slk}(C_2) = {\lk}(C_2,C_2^{\sat}) = -2$.
Other cases are treated similarly.

It is convenient to keep this data in the $3\times 3$ tabulation matrix $L$ whose  entries are
$l_{ij} = {\lk}(C_i,C_j)$ for $i \ne j$ and
$l_{ij} = {\slk}(C_i)$ for $i = j$:
\[ L = \begin{pmatrix}
-2 & -2 & -2 \\ -2 & -2 & -2 \\ -2 & -2 & -2
\end{pmatrix}. \]
Proposition~\ref{prop:Hel} implies
$\Hel(\tilde\xi_{\bt_0}) = -18\varphi^2$ and
$\Hel(\tilde\xi_{\bt_1}) = 18\varphi^2$.
We then have (for all $\bt \in \cT(\cR)$)
\begin{equation}
\label{equation:Helhex}
\Hel(\tilde\xi_{\bt}) = 36\varphi^2 \Tw(\bt) + C,
\qquad C = -18\varphi^2. 
\end{equation}
The value of the constant $C$ depends on two arbitrary choices made above:
the choice of base tiling and the choice of isolating shell.
\end{example}

\begin{remark}
\label{remark:strict}
Strictly speaking, the region $\cR$ in Example~\ref{example:hex}
must be massaged to make its boundary $\partial \cR$ a manifold, since the point $(1,1,1)$ is singular.
This can be easily achieved by considering the boundary $\partial \cR_\epsilon$ of 
a small $\epsilon$-neighborhood $\cR_\epsilon$ of $\cR$, as this will not affect the computations of 
$\Hel$ and $\Tw$.
Another solution is to include in $\cR$ a domino covering
either of the missing cubes, such as
$[1,3]\times[1,2]\times[0,1]$.
% and $[0,1]\times[0,1]\times[1,2]$.
The isolating shell needs to be slightly modified
but we arrive at the same conclusions.
\end{remark}

%%%%%%%%%%%%%%%%%%%%%%%%%%%%%%%%%%%%%%%%%%%%%%%%%%%%%%%%%%%%%%%%%%%%%%%

\section{Proof of  the main theorem on twist and helicity}
\label{section:proofmain}

Our aim is to prove Theorem~\ref{theo:main},
i.e., to relate $\Hel(\tilde\xi_{\bt})$ with the twist $\Tw(\bt)$   for a tiling of zero relative flux and satisfying
 the properties of Theorem~\ref{theo:Fluxtwist}.
We prove two lemmas in this direction.

\goodbreak

\begin{lemma}
\label{lemma:fliplemma}
Let $\cR$ be a cubiculated region with a fixed isolating shell.
Let $\bt_0, \bt_1$ be domino tilings of $\cR$
of zero relative flux and
let $\xi_{\bt_i}$ be the corresponding divergence-free vector fields.
\begin{enumerate}
\item{If $\bt_0$ and $\bt_1$ differ by a flip 
then $\Hel(\tilde\xi_{\bt_1}) = \Hel(\tilde\xi_{\bt_0})$.}
\item{If there is a positive trit from $\bt_0$ to $\bt_1$ 
then $\Hel(\tilde\xi_{\bt_1}) = \Hel(\tilde\xi_{\bt_0}) + 36\varphi^2$.}
\end{enumerate}
\end{lemma}

\goodbreak

\begin{proof}
To see the invariance under the flip (item 1) we start with $A$ being  the union
of the four unit cubes involved in the flip (as in Example~\ref{example:221}) and $B = M \smallsetminus A$ its any isolating shell. 
The fields $\xi_{\bt_i}$ are defined and different in $A$, but continued into $B$ as the same field $\zeta$. 
One can see that the relative helicities of $\tilde\xi_{\bt_i}=\xi_{\bt_i}+\zeta$ coincide, $\Hel(\tilde\xi_{\bt_1}) = \Hel(\tilde\xi_{\bt_0})$. (For instance, for the choice of continuation $\zeta$  in Figure~\ref{fig:flip5c1} both helicities vanish.) 
For a general region $\cR$ and a fixed isolating shell $B$ with $M:=\cR \cup B$
we still define $A$ as the  four  cubes ``participating in the flip" from $\xi_{\bt_0}$ to $\xi_{\bt_1}$, while regard its complement $B':=M\smallsetminus A$ as a new isolating shell. The same example shows  the invariance of relative helicity,  $\Hel(\tilde\xi_{\bt_1}) = \Hel(\tilde\xi_{\bt_0})$.

The change of relative helicity  under the trit (item 2) we define $A$
to be the union of the six unit cubes involved in the trit.
The rest of the proof is similar, except that we now use Example~\ref{example:hex}.
\end{proof}

%\footnote{BK: The previous version:
%\textcolor{blue}{Consider the first item. We begin with some notation.
%Let $A_{\old} = \cR$, $B_{\old}$ be the (fixed) isolating shell and $M = A_{\old} \cup B_{\old}$.
%The isolating shell also defines the vector field $\zeta_{\old}$ in the region $B_{\old}$.
%The vector fields $\xi_{\bt_i}$ are defined in $A_{\old}$. Let $A_{\flip} \subseteq \cR$ be the union
%of the four unit cubes involved in the flip and let $B_{\new} = M \smallsetminus A_{\flip}$.
%Let $\xi_{\bt_i,\flip}$ be the restriction of $\xi_{\bt_i}$ to $A_{\flip}$. Example~\ref{example:221} gives us a rather explicit
%vector field $\zeta_{\flip}$ on $B_{\new}$ and the equality
%\[ \Hel(\xi_{\bt_1,\flip} + \zeta_{\flip}) - \Hel(\xi_{\bt_0,\flip} + \zeta_{\flip}) = 0. \]
%The vector fields $\xi_{\bt_0}$ and $\xi_{\bt_1}$ coincide in $A_{\old} \smallsetminus A_{\flip} =
%B_{\new} \smallsetminus B_{\old}$.
%Define $\zeta_{\new}$ on $B_{\new}$ by adding   $\zeta_{\old}$ to the restriction of either $\xi_{\bt_i}$
%to $B_{\new} \smallsetminus B_{\old}$. We apply Theorem~\ref{theo:AB} 
%(with $A = A_{\flip}$ and $B = B_{\new}$) to the above equation to obtain
%\[ \Hel(\xi_{\bt_1,\flip} + \zeta_{\new}) - \Hel(\xi_{\bt_0,\flip} + \zeta_{\new}) = 0. \]
%Since $\tilde\xi_{\bt_i} = \xi_{\bt_i} + \zeta_{\old} = \xi_{\bt_i,\flip} + \zeta_{\new}$, we have 
%$\Hel(\tilde\xi_{\bt_1}) = \Hel(\tilde\xi_{\bt_0})$, as desired.}}

%\medskip
%For the second item of the proof above we can also use Example~\ref{example:332} and a slightly different construction of the set $A_{\new}$.

Our next lemma considers refinements.
We must first define how to refine an isolating shell.
Consider a cubiculated region $\cR$, a tiling $\bt$
and a valid isolating shell $(M,A,B,\zeta)$.
The boundary $\partial\cR$ is quadriculated:
in the middle of each square, $\zeta$ draws a pipe
with flux $\phi$, pointing in or out according to color.
That pipe is connected to somewhere in $\partial\cR$
to another square of the opposite color.

When we refine $\cR$ to obtain $\cR'$,
each old large square of $\partial\cR$ is decomposed
into $25$ new small squares of $\partial\cR'$.
Assume without loss of generality that the old pipe in $\zeta$
is now in the central small square.
Construct $12$ new short pipes
matching the other $24$ small squares
as in Figure~\ref{fig:partialrefine}.
(Note that different choices of the horizontal direction on a square face of $\partial\cR$
define equivalent isolating shells.)
This defines $\zeta'$ and the desired isolating shell for $\cR'$.

\begin{figure}[ht]
\begin{center}
\includegraphics[scale=0.2]{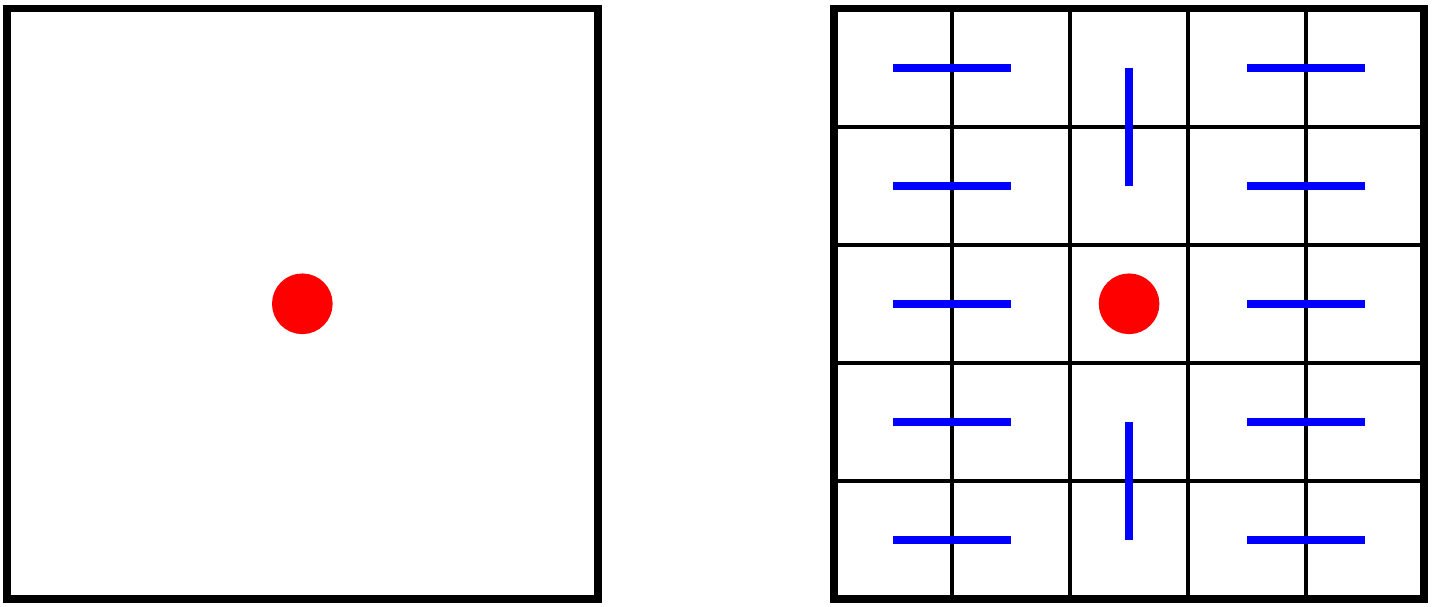}
\end{center}
\caption{A square in $\partial\cR$ is decomposed
into $25$ small squares in $\partial\cR'$.
We refine an isolating shell by adding short pipes
joining the centers of the unmatched new squares.}
\label{fig:partialrefine}
\end{figure}

\goodbreak

\begin{lemma}
\label{lemma:refinelemma}
Let $\cR$ be a cubiculated region with a fixed isolating shell.
Refine $\cR$ to obtain $\cR'$ and a corresponding isolating shell.
Let $\bt$ be a domino tiling of $\cR$ and $\bt'$ be its refinement,
a tiling of $\cR'$.
Let $\xi_{\bt}$ and $\xi_{\bt'}$
be the corresponding divergence-free vector fields.
We have $\Hel(\tilde\xi_{\bt}) = \Hel(\tilde\xi_{\bt'})$.
\end{lemma}

\goodbreak

\begin{proof}
We first perform a few flips on $\bt'$ to obtain a more drawable tiling
$\bt^\star$ of $\cR'$:
by Lemma~\ref{lemma:fliplemma}, $\Hel(\xi_{\bt^{\star}}) = \Hel(\xi_{\bt'})$.
The tiling $\bt^\star$ is drawn in Figure~\ref{fig:protorefine}.
Notice that each domino of $\bt$ is tiled by $125 = 5^3$
small dominoes of $\bt^\star$.

\begin{figure}[ht]
\begin{center}
\includegraphics[scale=0.25]{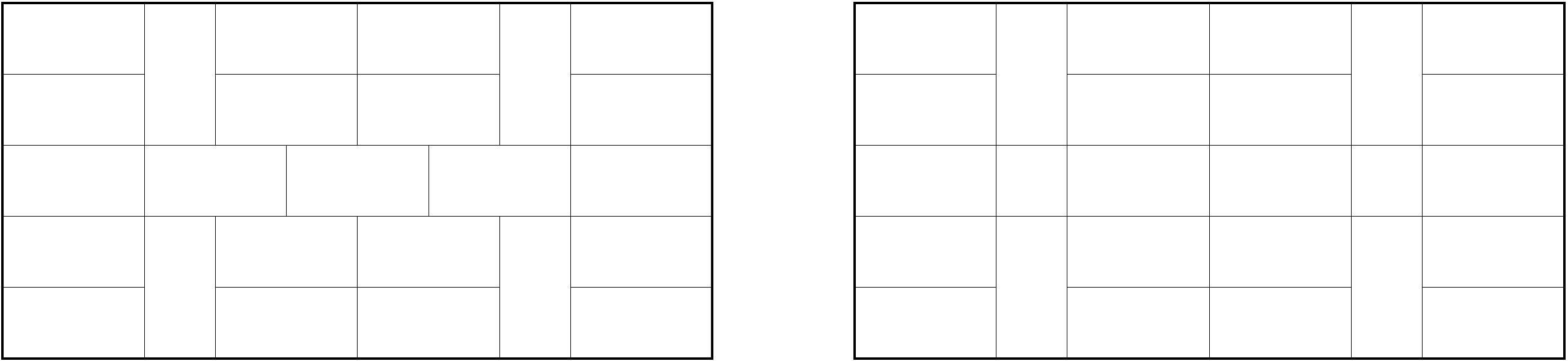}
\end{center}
\caption{The tiling $\bt^\star$,
a minor modification of the refinement $\bt'$.
We show here two floors of a domino of $\bt$.
The large domino is drawn horizontally.
At the left, we draw the central floor.
At the right, any of the other $4$ floors.}
\label{fig:protorefine}
\end{figure}

We now construct
the vector field $\xi_{\bt^\star}$, that is,
the pipe system corresponding to the tiling $\bt^\star$ as  in Figure~\ref{fig:refine}.

\begin{figure}[ht]
\begin{center}
\includegraphics[scale=0.3]{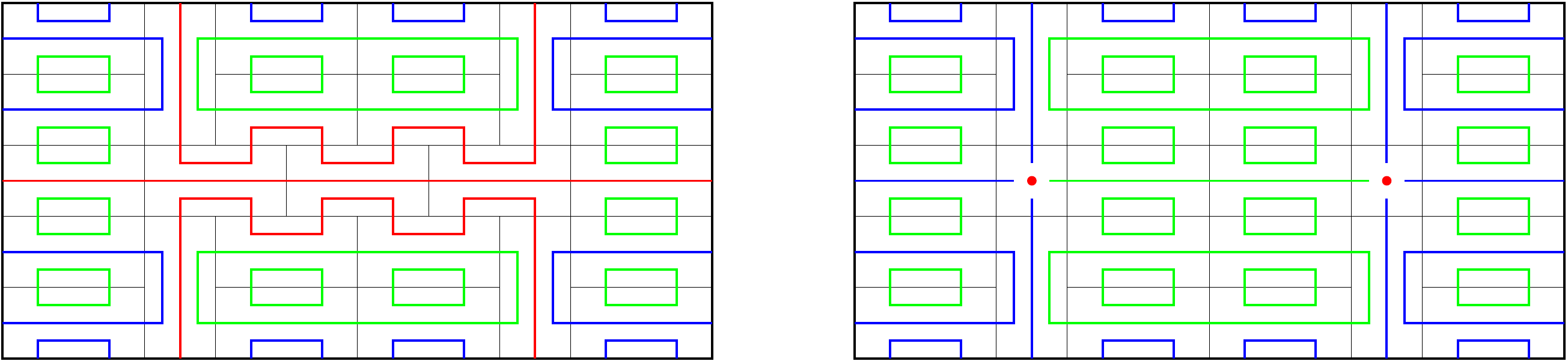}
\end{center}
\caption{The vector field $\xi_{\bt^\star}$.
The position is the same as in Figure~\ref{fig:protorefine}.
The pipes which were already present in $\xi_{\bt}$
are drawn in red and remain essentially unchanged.
The remaining pipes correspond to new curves, all trivial.}
\label{fig:refine}
\end{figure}

The old pipes (that is, the pipes which were already in $\xi_{\bt}$)
remain essentially as they were (in $\xi_{\bt}$),
with no change of linking or self-linking numbers.
The new pipes define only trivial links.
This implies $\Hel(\tilde\xi_{\bt}) = \Hel(\tilde\xi_{\bt^\star})$,
completing the proof.
\end{proof}

We are ready to state and prove our main result,
Theorem~\ref{theo:main}.
We restate it in a slightly different way.

\begin{theo} {\bf (=Theorem \ref{theo:main}$'$)}
Let $\cR$ be a cubiculated region. % (with a choice of a base tiling).
Let $\bt_0$ be a tiling of $\cR$ such that $\RFlux(\bt_0) = 0 \in H_1(\cR,\partial\cR)$.
% .
Construct an isolating shell for $\bt_0$.
There exists a contant $C \in \RR$ such that
\[ \Tw(\bt) = \frac{1}{36\varphi^2} \Hel(\tilde\xi_{\bt}) + C \]
for all $\bt \in \cT_{0}(\cR)$.
\end{theo}

\begin{proof}
Recall that given an initial  tiling $\bt_0$,  the set of all refinements of tilings  
$\cT_{0}(\cR)$ with the same flux is denoted by $\cT_{0}(\cR^{(\ast)})$.
Define $\widetilde\Tw: \cT_{0}(\cR^{(\ast)}) \to \RR$ by
\[ \widetilde\Tw(\bt) = \frac{1}{36\varphi^2} 
\left(\Hel(\tilde\xi_{\bt}) - \Hel(\tilde\xi_{\bt_0})\right). \]
 We claim that $\widetilde\Tw$ assumes integer values
and is   the  twist function, 
as in Theorem-Definition~\ref{theo:Fluxtwist}. Indeed, properties 1 and 2 
follow from items 1 and 2 of Lemma~\ref{lemma:fliplemma}, while its property 3
follows from Lemma~\ref{lemma:refinelemma}.
The fact that $\widetilde\Tw$ assumes integer values
now follows from the connectivity of graph $\cT_0(\cR^{(\ast)})$
(Theorem~\ref{theo:flipsandtrits}).
 Uniqueness of $\Tw$ up to an additive constant
(which also follows from Theorem~\ref{theo:flipsandtrits})
completes the proof.
\end{proof}

\begin{remark}
\label{cor:RFlux}
The above theorem also gives an alternative proof of Theorem-Definition~\ref{theo:Fluxtwist}.
Indeed, we constructed a twist function $\Tw$ assuming values in $\ZZ$.
In \cite{FKMS} also the setting with $m\not=0$ is discussed from a combinatorial viewpoint,
where $m$ is related to the relative flux of the initial tiling.
One can prove that $m = 0$ if and only if $\RFlux(\bt_0) = 0$,
but we do not discuss this here as the proof uses different (combinatorial) tools.
\end{remark}

%%%%%%%%%%%%%%%%%%%%%%%%%%%%%%%%%%%%%%%%%%%%%%%%%%%%%%%%%%%%%%%%%%%%%%%%%%%%

\section{A 3D `crosses-and-toes' example}
\label{section:extra}

In this section we consider yet one more, larger example of tiling, which is of interest by itself.
One can use this example
and a slightly different construction of the set $A_{\new}$ for the proof of the second part of Lemma \ref{lemma:fliplemma}.

\begin{example}
\label{example:332}
Let $\cR = [0,3]\times[0,3]\times[0,2]$,
the $3\times3\times2$ box,
the same shown in Figure~\ref{fig:trit}.
A valid isolating shell is shown in Figure~\ref{fig:332-is}.
Figure~\ref{fig:332-tv} shows $\tilde\xi$ (i.e., the tubes)
for the fifth tiling in Figure~\ref{fig:trit},
the one with nine vertical dominoes.
Even though there are many curves, they are all trivial
and helicity is again zero.

\begin{figure}[ht]
\begin{center}
\includegraphics[scale=0.175]{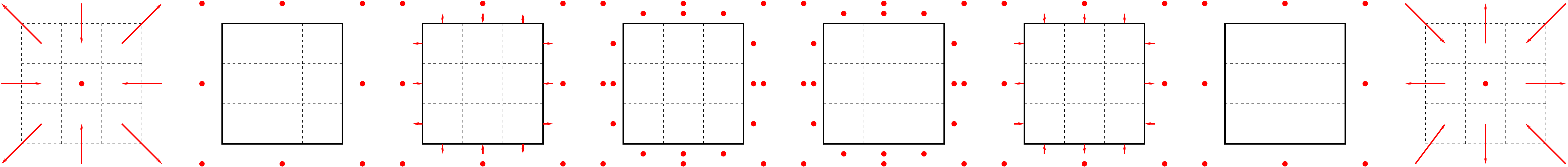}
\end{center}
\caption{A valid isolating shell for the $3\times3\times2$ box.
We show the following planes: $-\frac14$, $\frac14$, $\frac12$,
$\frac34$, $\frac54$, $\frac32$, $\frac74$ and $\frac94$.
The red dots are vertical tubes.
There is yet another central vertical tube, not entirely shown.
}
\label{fig:332-is}
\end{figure}

\begin{figure}[ht!]
\begin{center}
\includegraphics[scale=0.175]{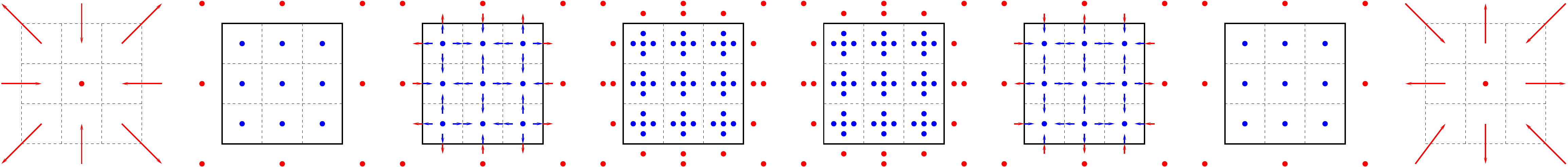}
\end{center}
\caption{The five-tube version of the vertical tiling 
of the $3\times 3\times 2$ box (Example~\ref{example:332}).}
\label{fig:332-tv}
\end{figure}

Let $\bt_0$ be the first tiling in Figure~\ref{fig:trit},
the one with nonzero twist, so that $\Tw(\bt_0) = -1$. 
Figure~\ref{fig:332-t0} shows $\bt_0$ again;
on the second row, the tubes for $\xi_{\bt_0}$.
The other rows show one curve at a time:
call them $C_1$ to $C_6$, from top to bottom.
There is a central vertical curve $C_7$, not shown.

\begin{figure}[p]
\begin{center}
\includegraphics[scale=0.175]{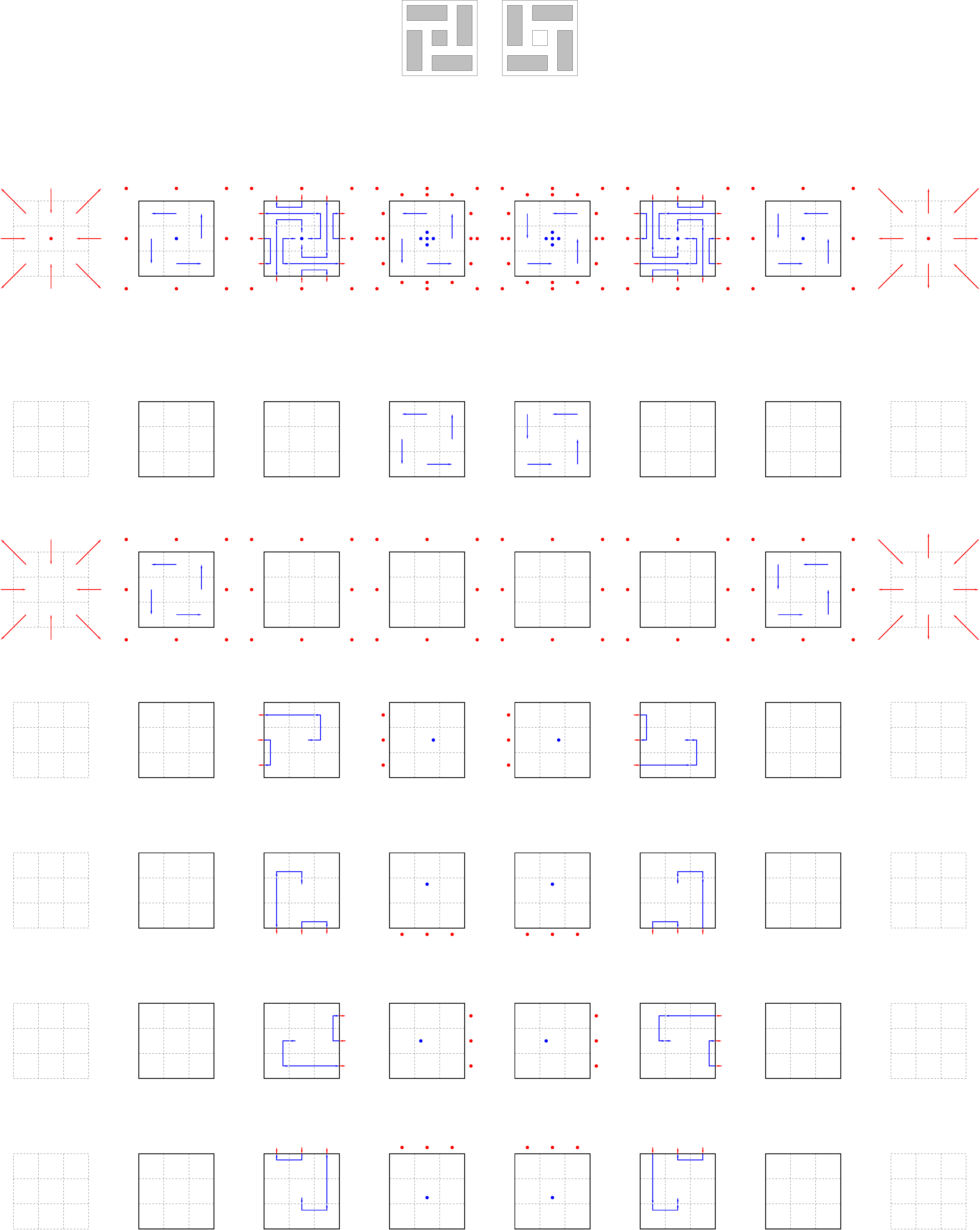}
\end{center}
\caption{On the first row, the tiling $\bt_0$
from Example~\ref{example:332}.
On the second, the vector field $\xi_{\bt_0}$,
i.e., the five-tube version of $\bt_0$.
Curves are shown separately on the other rows.}
\label{fig:332-t0}
\end{figure}

Computing linking and self-linking numbers
by using signed  intersections of curve projections is not hard,
but laborious, and we omit the details for now.
The results are shown in a $7 \times 7$ tabulation matrix $L$ with entries 
$l_{ij} = {\lk}(C_i,C_j)$ for $i \ne j$ and
$l_{ij} = {\slk}(C_i)$ for $i = j$:
\[ L = \begin{pmatrix}
 0 &  0 & -1 & -1 & -1 & -1 & -1 \\
 0 &  0 &  0 &  0 &  0 &  0 & -1 \\
-1 &  0 & -1 & -1 & -1 & -1 & -1 \\
-1 &  0 & -1 & -1 & -1 & -1 & -1 \\
-1 &  0 & -1 & -1 & -1 & -1 & -1 \\
-1 &  0 & -1 & -1 & -1 & -1 & -1 \\
-1 & -1 & -1 & -1 & -1 & -1 & 0 
\end{pmatrix}; \]
Equation~\eqref{equation:Hel} then implies
$\Hel(\xi_{t_0}) = -36\varphi^2 = 36\varphi^2 \Tw(\bt_0)$.

Note that our box $\cR$ has precisely $229$ tilings.
These are $\bt_0$ (with twist $-1$),
the mirror image of $\bt_0$
(on the plane $z=1$, for instance; this tiling has twist $+1$)
and another $227$ tilings of twist $0$.
In this example, any two tilings of twist $0$
can be joined by a finite sequence of flips.
It can be verified (and it follows from Theorem~\ref{theo:main})
that
\begin{equation}
\label{equation:Hel332}
\Hel(\xi_{\bt}) = 36\varphi^2 \Tw(\bt)
\end{equation}
for all $\bt \in \cT(\cR)$.
If $\bt_1$ is the last tiling in Figure~\ref{fig:trit}
it is easy to verify that $\Hel(\xi_{\bt_1}) = 0$
(all curves are trivial; see Figure~\ref{fig:332-tv}).
\end{example}

%%%%%%%%%%%%%%%%%%%%%%%%%%%%%%%%%%%%%%%%%%%%%%%%%%%%%%%%%%%%%%%%%%%%%%%

%%%%%%%%%%%%%%%%%%%%%%%%%%%%%

\section{Appendix: Helicity via linking and self-linking}
\label{appendix}

Here we recall the derivation of the explicit formula \eqref{equation:Hel}:

\begin{prop}
For the helicity of a divergence-free field $\xi$ confined to several narrow non-intersecting linked flux pipes $\cup\, T_i$ in a simply-connected manifold $M$:
$$
{\rm Hel}(\xi ) = 2\sum_{i<j}{\lk} (C_i,C_j)\cdot {\Flux}_i\cdot{\Flux}_j + \sum_i {\slk}(C_i)\cdot ({\Flux}_i)^2.
$$
\end{prop}

Here closed curves  $C_i$ are cores of the pipes $T_i$ and ${\rm Flux}_i$ are the fluxes in those pipes, while ${\lk} (C_i,C_j)$
is the Gauss linking number of the curves.
The self-linking number ${\slk}(C_i)$ is the ``average self-linking" of trajectories of $\xi$  in the given pipe. Namely, if all trajectories of the field $\xi$ in a pipe are closed and have the same pairwise linking
(for instance, the Poincar{\'e} map in the pipe cross-section is a solid rotation  by a rational angle), then ${\slk}$ is equal to that pairwise linking number. For an arbitrary field, the self-linking number slk must be replaced by the integral of the asymptotic winding number of the field's trajectories introduced by Fathi  \cite{Fathi1980} and discussed below.

\smallskip

This formula was a folklore statement since Moffatt's discovery
in \cite{Moffatt69} of the formula 
$$
{\rm Hel}(\xi ) = 2~{\lk} (C_1,C_2)\cdot {\rm Flux}_1\cdot {\rm Flux}_2
$$
for helicity ${\rm Hel}(\xi)=
\int \limits_{M} (\xi , {\rm{curl}}^{-1}\xi )\; d^3x$
of the field $\xi$ supported in two tubes $T_1\cup T_2\subset \RR^3$
with all orbits closed, with equal periods,
and having no internal twist inside the tubes,
and since Arnold's generalization of it
to arbitrary divergence-free fields
as the asymptotic linking number of its trajectories.
It explicitly appeared e.g. in \cite{Scheeler2017} 
with slightly stronger assumptions and
a special interpretation of \textit{slk} via the Calugareanu invariant,
cf. \cite{Moffatt-Ricca}. 
%Note that one can drop the absolute value sign of $Flux_i$ since $\lk$ also changes sign with orientation of each curve. 

The derivation of the formula above can be obtained
as a following sequence of   statements.

\medskip

\begin{proof}
1) According to Arnold's theorem \cite{Arnold73} the total helicity integral
${\rm Hel}(\xi)=\int \limits_{M} (\xi , {\rm{curl}}^{-1}\xi )\; d^3x$
is equal to the average linking number of trajectories,
$$
{\rm Hel}(\xi)=\iint_{M\times M} {\lk}(x,y)\mu_x \mu_y,
$$
where $ {\lk}(x,y)$ is the asymptotic linking of the trajectories $g^s(x)$ and $g^t(y)$ of the field $\xi$ starting at the points $x,y\in M$, and $\mu$ is the  volume form on $M$. For a field supported on the union of pipes $\cup \,T_i$ it splits into pairwise integrals.

\medskip

2) For two tubes $T_i$ and $T_j$ carrying  a field $\xi$, all whose  orbits are closed with equal periods and with no internal twist, their cross-linking  is given by Moffatt's formula. Note that this also holds in the more general setting of not necessarily closed trajectories and for arbitrary internal twist inside the tubes.
Indeed, by taking the average linking of one closed curve with  trajectories in the other tube   (the case studied in \cite{Arnold73}),
 the corresponding linking is proportional to the flux through a Seifert surface for that curve and it does not depend in the internal twist in the tube. Hence in the formula for cross-linking two  tubes $T_i$ and $T_j$,  only the mutual linking of their cores $C_i$ with $C_j$ and the product of their fluxes $\Flux_i\cdot \Flux_j$ appears.

3) What remains to prove is that the helicity for a field inside the $i$th tube is given by ${\slk}(C_i)\cdot ({\rm Flux}_i)^2.$ 
This is essentially the statement from \cite{Gambaudo-Ghys}, see also \cite{Shelukhin}. 

In more detail, use some Riemannian metric on $M$ to introduce trivialization along the tube, which now can be regarded as a solid torus $D^2\times S^1$ with globally defined cylindrical coordinates $(r, \phi, z)$, where $r\ge 0$ and $z ~{\rm mod}\, 1$. (The invariants introduced will not depend on that choice.) Following \cite{Fathi1980}
for any pair of points $x,y\in D$ one considers the asymptotic winding number 
$$
W_\phi(x,y)=\lim_{T, S\to \infty} \frac{\phi(g^t(y))-\phi(g^s(x))}{2\pi\cdot T\cdot S}
$$ 
for a pair of trajectories $g^t(y), \,t\in [0, T]$ and $g^s(x),\, s\in [0, S]$. It is defined for almost all pairs $x,y\in D$ and it is an integrable function
on $D\times D$. It  was proved in \cite{Gambaudo-Ghys, Shelukhin} that its integral over $D\times D$ is equal to the helicity invariant of the field. It is also the Calabi invariant of the corresponding Poincar{\'e} map $\Psi:D\to D$ for the flow of $\xi$ in the tube:
$$
{\rm Hel}(\xi)={\rm Cal}(\Psi)=\iint_{D\times D} W_\phi(x,y) \,\omega_x\wedge \omega_y
$$
where $\omega$ is the invariant area form induced by the divergence-free vector field $\xi$ in the pipe cross-section, 
$\omega_x=i_\xi\mu_x$.
Note that for a field $\xi$ in a tube along $C$ with closed trajectories of period 1 the value of $W_\phi(x,y)$ is constant and $W_\phi(x,y)={\slk}\,(C)$. On the other hand, the integral of the induced area form on the cross-section is equal to $\Flux_\xi$ of the corresponding vector field $\xi$
through that cross-section. 
\end{proof}
%$\diamondsuit$

\medskip

\begin{remark} 
Note that the above general formula is, in particular, valid for the case of closed curves with {\it different periods}. 
For instance, rescale  a divergence-free field $\xi$ as $f\xi$ for a function $f$ on a tube $T$.
The divergence-free constraint on $f\xi$ implies that $f$ is a first integral of the field $\xi$, while 
the flux ${\rm Flus}_{f\xi}$ of the new field must be the same in all cross-sections. 
Then  ${\rm Hel}(f\xi)=(\int_D f(x)\,\omega_x)^2\cdot {\rm Hel}(\xi)$, see e.g.  \cite{Gambaudo-Ghys}, while the latter expression transforms as follows:
 $$
{\rm Hel}(f\xi)=
%(\int_D f(x)\,\omega_x)^2\cdot {\rm Hel}(\xi) = 
(\int_D f(x)\,\omega_x)^2\cdot {\slk}(C)\cdot ({\rm Flux}_\xi)^2
=  {\slk}(C)\cdot ({\rm Flux}_{f\xi})^2\,,
  $$
which emphasizes  universality of the formula \eqref{equation:Hel} via fluxes.
\end{remark}

%%%%%%%%%%%%%%%%%%%%%%%%%%%%%%%%%%%%%%%%%%%%%%%%%%%%%%%%%%%%%%%%%%%%%%%

\medskip

\noindent
\footnotesize
B.K.: Department of Mathematics, University of Toronto \\
40 St. George Street, Toronto, ON M5S 2E4, Canada \\
\url{khesin@math.toronto.edu}
\medskip

\noindent
\footnotesize
N.S.: Departamento de Matem\'atica, PUC-Rio \\
Rua Marqu\^es de S\~ao Vicente 225, Rio de Janeiro, RJ 22451-900, Brazil \\
\url{saldanha@puc-rio.br}

\end{document}